\documentclass[reqno,10pt]{amsart}

\usepackage[foot]{amsaddr}
\usepackage{graphicx,enumerate,nicefrac,bm}
\usepackage[dvipsnames]{xcolor}
\usepackage{cite}
\usepackage{amsmath, amssymb}
\usepackage{tikz}\usetikzlibrary{matrix}
\usepackage{a4wide}
\usepackage{subfig}
\usepackage{changes}
\usepackage{mathtools}
\usepackage{booktabs}
\usepackage{algorithm,algpseudocode,tabularx}
\usepackage{comment}
\usepackage{stmaryrd}

\makeatletter
\newcommand{\multiline}[1]{%
  \begin{tabularx}{\dimexpr\linewidth-\ALG@thistlm}[t]{@{}X@{}}
    #1
  \end{tabularx}
}
\makeatother

\newcommand{\ukappa}{\boldsymbol{\kappa}}
\newcommand{\udelta}{\boldsymbol{\delta}}

\newcommand{\uv}{\mathbf{u}}

\newcommand{\pnl}{P_N^{n,\ell}}
\newcommand{\pnll}{P_N^{n_N,\ell_N}}
\newcommand{\pn}{P_N^n}
\newcommand{\Uv}{\mathbf{U}}
\newcommand{\uvn}{\Uv_N^n}
\newcommand{\uvnl}{\Uv_N^{n,\ell}}
\newcommand{\uvnll}{\Uv_N^{n_N,\ell_N}}
\newcommand{\uvnlo}{\Uv_N^{n,\ell+1}}
\newcommand{\Dv}{\mathbf{D}}
\newcommand{\Du}{\D \Uv}
\newcommand{\Vv}{\mathbf{V}}
\newcommand{\Wv}{\mathbf{W}}
\newcommand{\fv}{\mathbf{f}}
\newcommand{\vv}{\mathbf{v}}
\newcommand{\wv}{\mathbf{w}}
\newcommand{\Sv}{\mathbf{S}}

\newcommand{\A}{\mathcal{A}}
\newcommand{\Ao}{\mathrm{A}}
\newcommand{\B}{\mathcal{B}}
\newcommand{\D}{\mathrm{D}}
\newcommand{\J}{\mathrm{J}}
\renewcommand{\H}{\mathsf{H}}

\newcommand{\F}{\mathrm{F}}

\newcommand{\V}{\mathbb{V}}
\newcommand{\Q}{\mathbb{Q}}
\newcommand{\T}{\mathcal{T}}
\newcommand{\To}{\mathrm{T}}
\newcommand{\tn}{\To_N^n}
\newcommand{\R}{\mathcal{R}}
\renewcommand{\D}{\mathrm{D}}
\newcommand{\x}{\bm{x}}
\newcommand{\dx}{\,\mathsf{d}\bm{x}}

\newcommand{\Rsym}{\mathbb{R}^{d\times d}_{\mathrm{sym}}}
\newcommand{\jdiv}{\mathcal{J}_{\dive}^N}
\newcommand{\jq}{\mathcal{J}_{\Q}^N}
\newcommand{\Rnum}[1]{\mathrm{\uppercase\expandafter{\romannumeral #1\relax}}}
\newcommand{\Cb}{C_{\mathcal{B}}}

\newcommand{\norm}[1]{\left\|#1\right\|}
\newcommand{\nnn}[1]{{\left\vert\kern-0.25ex\left\vert\kern-0.25ex\left\vert #1 
    \right\vert\kern-0.25ex\right\vert\kern-0.25ex\right\vert}}

\newcommand{\dprod}[1]{\langle #1\rangle}
\newcommand{\jmp}[1]{\left\llbracket#1\right\rrbracket}

\DeclareMathOperator*{\esssup}{ess\,sup}

\DeclareMathOperator*{\dive}{div}

\newtheorem{theorem}{Theorem}[section]
\newtheorem{lemma}[theorem]{Lemma}
\newtheorem{proposition}[theorem]{Proposition} 
\newtheorem{corollary}[theorem]{Corollary}

\theoremstyle{definition}

\newtheorem{remark}[theorem]{Remark}
\newtheorem{assumption}[theorem]{Assumption}

\title[AILFEM for Bingham fluids]{Adaptive iterative linearised finite element methods for implicitly constituted incompressible fluid flow problems and its application to Bingham fluids}

\author[P.~Heid]{Pascal Heid}
\email{pascal.heid@maths.ox.ac.uk}

\author[E.~S\"{u}li]{Endre S\"{u}li}
\email{endre.suli@maths.ox.ac.uk}

\address{Mathematical Institute, University of Oxford, Woodstock Road, Oxford OX2 6GG, UK}

\thanks{PH acknowledges the financial support of the Swiss National Science Foundation (SNF), Project No. P2BEP2\underline{\space}191760.}



\begin{document}

\begin{abstract}
In this work, we introduce an iterative linearised finite element method for the solution of Bingham fluid flow problems. The proposed algorithm has the favourable property that a subsequence of the sequence of iterates generated converges weakly to a solution of the problem. This will be illustrated by two numerical experiments.
\end{abstract}



\maketitle

\section{Introduction}

In this work, we consider implicitly constituted incompressible fluid flow problems, as introduced by Rajagopal in~\cite{Rajagopal:03, Rajagopal:06}: Instead of demanding, as in the classical theory of continuum mechanics, that the Cauchy stress is an explicit function of the symmetric part of the gradient of the velocity vector, one may allow an implicit relation of those two quantities. For a rigorous mathematical analysis, including the proof of the existence of a weak solution for models of implicitly constituted fluids we refer the reader to~\cite{Bulicek:09} and~\cite{Bulicek:12} for  steady and unsteady flows, respectively. Regarding the numerical analysis of such fluid flow models, very few results have been published so far. In~\cite{DiKrSu:2013} the authors proved the weak convergence of a sequence of finite element approximations to a weak solution in the steady case. An a posteriori analysis of finite element approximations of the model was carried out in~\cite{KreuzerSuli:2016}; in addition, those authors proved the weak convergence of an adaptive finite element method. We further refer to~\cite{Hron:17} for a numerical investigation of several finite element discretisation methods and linearisation schemes for Bingham and stress-power-law fluids. Just recently, in~\cite{Orozco:21}, the author proposed a semismooth Newton method for the numerical approximation of a steady solution to implicit constitutive fluid flow problems. For the numerical analysis of the unsteady case (not considered in this work) we refer to~\cite{SuliTscherpel:19, FarrellSuliOrozco:20}.

The aim of this work is to extend the theory of~\cite{KreuzerSuli:2016} to adaptive \emph{iterative linearised} finite element methods for implicitly constituted fluid flow problems; i.e., in contrast to~\cite{KreuzerSuli:2016}, we will take into account the approximation of a finite element solution by a (linear) iteration scheme. Subsequently, it shall be shown that this abstract analysis can be applied in the context of Bingham fluids.

The organisation of the paper is as follows: In Section 2, we present the preliminaries, including the formulation of the problem, its finite element approximation, as well as an a posteriori error analysis. In Section 3, we will state the adaptive iterative linearised finite element method (AILFEM), and prove the convergence of this algorithm. Since this result can be shown by some minor modifications of the analysis in~\cite{KreuzerSuli:2016}, we will omit the details of the proof, and rather give a rough sketch in Appendix~\ref{app:1} and highlight where the main differences occur. Next, in Section 4, we will verify the assumptions required for the convergence of AILFEM for Bingham flows. We will perform corresponding numerical experiments in Section 5, and conclude our work with some closing remarks.

\section{Preliminaries}     

In this section, we will introduce the model of a steady flow of an incompressible fluid in a bounded open Lipschitz domain $\Omega \subset \mathbb{R}^d$, $d \in \{2,3\}$, with polyhedral boundary $\partial \Omega$, which satisfies an implicit constitutive relation given by a maximal monotone $r$-graph. Beforehand, let us introduce some basic notions concerning Lebesgue and Sobolev spaces.\\

\subsection{Basic notations}

For any measurable open set $\omega \subseteq \Omega$  and $s \in [1,\infty)$ we denote by $L^s(\omega):=L^s(\omega;\mathbb{R})$ the Lebesgue space of $s$-integrable functions with corresponding norm $\norm{f}_{s,\omega}:=\left(\int_\omega |f(\x)|^s \dx\right)^{\nicefrac{1}{s}}$. Moreover, $L^\infty(\omega)=L^\infty(\omega;\mathbb{R})$ denotes the Lebesgue space of essentially bounded functions with the norm $\norm{f}_{\infty,\omega}:=\esssup_{\x \in \omega}|f(\x)|$, and $L_0^s(\omega):=\{f \in L^s(\omega): \int_\omega f \dx=0\}$ denotes the set of functions (in the corresponding Lebesgue space) with zero mean value. We note that, for $s \in (1,\infty)$, $L^{s'}(\omega)$ and $L^{s'}_0(\omega)$ are the dual spaces of $L^s(\omega)$ and $L^{s}_0(\omega)$, respectively, where $s' \in (1,\infty)$ is the H\"{o}lder conjugate of $s$, i.e., $\nicefrac{1}{s}+\nicefrac{1}{s'}=1$.

Likewise, for $s \in [1,\infty]$, we denote by $W^{1,s}(\omega)^d:=W^{1,s}(\omega;\mathbb{R}^d)$ the space of vector-valued Sobolev functions. Moreover, for $s \in [1,\infty]$, the space of vector-valued Sobolev functions with zero trace along the boundary is denoted by $W_0^{1,s}(\omega)^d$ and is equipped with the norm $\|\fv\|_{1,s,\omega}:=\norm{\nabla \fv}_{s,\omega}$. Equivalently, for $s \in [1,\infty)$, $W_0^{1,s}(\omega)^d$ is the closure in $W^{1,s}(\omega)^d$ of $\mathcal{D}(\omega)^d:=C_0^\infty(\omega)^d$, i.e., the space of smooth (vector-valued) functions with compact support in $\omega$, and its dual space, for any $s \in (1,\infty)$, is denoted by $W^{-1,s'}(\omega)^d$. The norm in the dual space $W^{-1,s'}(\omega)^d$ is as usual given by 
\begin{align*}
\norm{\varphi}_{-1,s',\omega}:=\sup_{\stackrel{\vv \in W_0^{1,s}(\omega)^d}{\norm{\vv}_{1,s,\omega} =1}} \dprod{\varphi,\vv},
\end{align*}
where $\dprod{\cdot,\cdot}$ signifies the duality pairing. In the sequel, for $\omega=\Omega$, we omit the domain in the subscript of the norms; e.g., we write $\norm{\cdot}_{s}:=\norm{\cdot}_{s,\Omega}$.

Finally, for $\udelta,\ukappa \in \Rsym$ we denote by $\udelta:\ukappa$ the Frobenius inner-product, and by $|\ukappa|$ the Frobenius norm.

\subsection{Problem formulation}
As before, let $\Omega \subset \mathbb{R}^d$, $d \in \{2,3\}$, be a bounded open Lipschitz domain with polyhedral boundary $\partial \Omega$. For $r \in (1,\infty)$, we set
\begin{align*} 
\tilde{r}:=\begin{cases} \frac{dr}{2(d-r)} &\text{if } r \leq \frac{3d}{d+2}, \\ r' & \text{otherwise}, \end{cases}
\end{align*} 
where $r' \in (1,\infty)$ denotes the H\"older conjugate of $r$. Then, the implicitly constituted incompressible fluid flow problem under consideration reads as follows: for $\fv \in L^{r'}(\Omega)^d$ find $(\uv,p,\Sv) \in W_0^{1,r}(\Omega)^d \times L_0^{\tilde{r}}(\Omega) \times L^{r'}(\Omega)^{d \times d}$ such that
\begin{subequations} \label{eq:pde}
\begin{align}
\dive(\uv \otimes \uv+p \mathbf{I}-\Sv)&=\fv && \text{in } \mathcal{D}'(\Omega)^d, && \\
\dive \uv&=0 && \text{in } \mathcal{D}'(\Omega), &&\\
(\mathrm{D}\uv(\x),\Sv(\x)) & \in \mathcal{A}(\x) && \text{for a.e. } \x \in \Omega; &&
\end{align}
\end{subequations} 
here, $\mathrm{D}\uv:=\frac{1}{2}(\nabla \uv +( \nabla \uv)^{\mathrm{T}}) \in \mathbb{R}^{d \times d}_{\mathrm{sym}}:=\{\ukappa \in \mathbb{R}^{d \times d}:\ukappa=\ukappa^{\mathrm{T}}\}$ signifies the symmetric velocity gradient. Moreover, $\mathcal{A}:\Omega \to \mathbb{R}^{d\times d}_{\mathrm{sym}} \times \mathbb{R}^{d\times d}_{\mathrm{sym}}$ is a maximal monotone $r$-graph, i.e., for almost every $\x \in \Omega$ the following properties are satisfied:
\begin{enumerate}[(A1)]
\item[(A1)] $(0,0) \in \A(\x)$;
\item[(A2)] For all $(\ukappa_1,\udelta_1),(\ukappa_2,\udelta_2) \in \A(\x)$,
\begin{align*}
(\udelta_1-\udelta_2):(\ukappa_1-\ukappa_2) \geq 0;
\end{align*}
\item[(A3)] If $(\ukappa,\udelta) \in \Rsym \times \Rsym$ and 
\begin{align*}
(\overline{\udelta}-\udelta):(\overline{\ukappa}-\ukappa) \geq 0 \qquad \text{for all} \ (\overline{\ukappa},\overline{\udelta}) \in \A(\x),
\end{align*}  
then $(\ukappa,\udelta) \in \A(\x)$;
\item[(A4)] There exist a non-negative function $m \in L^1(\Omega)$ and a constant $c>0$, such that for all $(\ukappa,\udelta) \in \A(\x)$ we have that
\begin{align*}
\udelta:\ukappa \geq -m(x)+c(|\ukappa|^r+|\udelta|^{r'});
\end{align*}
\item[(A5)] The set-valued mapping $\A:\Omega \to \Rsym \times \Rsym$ is measurable, i.e., for any closed sets $\mathcal{C}_1,\mathcal{C}_2 \subset \Rsym$ we have that 
\[
\{\x \in \Omega: \A(\x) \cap (\mathcal{C}_1 \times \mathcal{C}_2) \neq \emptyset\}
\]
is a Lebesgue measurable subset of $\Omega$.
\end{enumerate}

The existence of a (not necessarily unique) solution to~\eqref{eq:pde} for $r>\frac{2d}{d+2}$ was first established in~\cite{Bulicek:09}. The proof is based on a (monotone) measurable selection $\Sv^\star:\Omega \times \Rsym \to \Rsym$ of the graph $\mathcal{A}$; i.e., for all $\mathbf{\ukappa} \in \Rsym$ we have that $(\mathbf{\ukappa},\mathbf{S}^\star(\x,\mathbf{\ukappa})) \in \mathcal{A}(\x)$ for almost every $\x \in \Omega$. In turn, the measurable selection $\Sv^\star:\Omega \times \Rsym \to \Rsym$ was approximated by a sequence of strictly monotone mappings $\Sv^n:\Omega \times \Rsym \to \Rsym$, $n = 0,1,2,\dotsc$, obtained by a mollification of $\Sv^\star$. In this work, following~\cite{KreuzerSuli:2016}, we allow for more general graph approximations.

\begin{assumption} \label{as:graphapp}
For any $n \in \mathbb{N}$ there exists a mapping $\Sv^n:\Omega \times \Rsym \to \Rsym$ such that
\begin{itemize}
\item $\Sv^n(\cdot,\ukappa):\Omega \to \Rsym$ is measurable for all $\ukappa \in \Rsym$;
\item $\Sv^n(\x,\cdot):\Rsym \to \Rsym$ is continuous for almost every $\x \in \Omega$;
\item $\Sv^n$ is strictly monotone in the sense that, for all $\ukappa_1 \neq \ukappa_2 \in \Rsym$, we have that
\begin{align*}
(\Sv^n(\x,\ukappa_1)-\Sv^n(\x,\ukappa_2)):(\ukappa_1-\ukappa_2)>0 \qquad \text{for almost every } \x \in \Omega;
\end{align*}
\item There exist constants $\tilde{c}_1,\tilde{c}_2 >0$ and non-negative functions $\tilde{m} \in L^1(\Omega)$ and $\tilde{k} \in L^{r'}(\Omega)$ such that, uniformly in $n \in \mathbb{N}$,
\begin{align}
|\Sv^n(\x,\ukappa)| & \leq \tilde{c}_1 |\ukappa|^{r-1}+\tilde{k}(\x), \label{eq:boundedS}\\
\Sv^n(\x,\ukappa):\ukappa & \geq \tilde{c}_2|\ukappa|^r-\tilde{m}(\x) \label{eq:monotoneS}
\end{align}
for all $\ukappa \in \Rsym$ and almost every $\x \in \Omega$.
\end{itemize}
\end{assumption}

Of course, we also need that $\Sv^n$ approximates, in a certain sense, the measurable selection $\Sv^\star$ of $\A$; this will be made precise in Section~\ref{sec:ailfem}. Then, the regularised counterpart of problem~\eqref{eq:pde} is given as follows: for $\fv \in L^{r'}(\Omega)^d$ find $(\uv^n,p^n) \in W_0^{1,r}(\Omega)^d \times L_0^{\tilde{r}}(\Omega)$ such that
\begin{alignat*}{2}
\dive(\uv^n \otimes \uv^n+p^n \mathbf{I}-\Sv^n(\uv^n))&=\fv & \quad & \text{in } \mathcal{D}'(\Omega)^d, \\
\dive \uv^n&=0 && \text{in } \mathcal{D}'(\Omega).
\end{alignat*}
We note that in the problem formulation above and in the following, the explicit dependence of $\Sv^n$ on $\x \in \Omega$ will be suppressed.

%

%

\subsection{Finite element spaces}

In this work, we consider a sequence $\{\T_N\}_N$ of shape-regular conforming triangulations of $\Omega$, such that $\T_{N+1}$ is obtained by a refinement of $\T_N$. For any $m \in \mathbb{N}$ let us denote by $\mathbb{P}_m$ the space of polynomials of degree at most $m$. Then, the corresponding conforming finite element spaces are given by 
\begin{align*}
\V(\T_N)&:=\{\Vv \in W_0^{1,2}(\Omega)^d: \Vv|_K \in \mathbb{P}_{\V}(K) \ \text{for all} \ K \in \T_N\},\\
\Q(\T_N)&:=\{Q \in L^\infty(\Omega): Q|_K \in \mathbb{P}_\Q(K) \ \text{for all} \ K \in \T_N\},
\end{align*} 
where $\mathbb{P}_\V$ and $\mathbb{P}_\Q$ are spaces of polynomials such that $\mathbb{P}_1^d \subseteq \mathbb{P}_\V \subseteq \mathbb{P}^d_{i^\star}$ and $\mathbb{P}_0 \subseteq \mathbb{P}_\Q \subseteq \mathbb{P}_{j^\star}$ for some $i^\star \geq j^\star \geq 0$. We note the nestedness $\V(\T_N) \subseteq \V(\T_{N+1})$ and $\Q(\T_N) \subseteq \Q(\T_{N+1})$ of the finite element spaces. We further introduce the space of discretely divergence-free velocity vectors
\begin{align*}
\V_0(\T_N):=\left\{\Vv \in \V(\T_N): \int_\Omega Q \dive \Vv \dx=0 \ \text{for all} \ Q \in \Q(\T_N)\right\},
\end{align*}
and the space 
\begin{align*}
\Q_0(\T_N):=\left\{Q \in \Q(\T_N): \int_\Omega Q \dx=0\right\}.
\end{align*}
In order to apply the analysis from~\cite{KreuzerSuli:2016}, we need to impose the same assumptions on the velocity-pressure pairs of finite element spaces. In the following, for any $K \in \mathcal{T}_N$, \[\mathcal{U}^N(K):= \cup\{K' \in \mathcal{T}_N: \overline{K} \cap \overline{K'} \neq \emptyset\}\] denotes the patch of (not necessarily facewise) neighbours of $K$.

\begin{assumption} \label{as:jdiv}
We assume that for each $N \in \mathbb{N}$ there exists a linear projection operator $\jdiv:W_0^{1,1}(\Omega)^d \to \V(\T_N)$ such that, for all $s \in [1,\infty)$,
\begin{itemize}
\item $\jdiv$ preserves divergence in the dual of $\Q(\T_N)$; i.e., for $\vv \in W_0^{1,s}(\Omega)^d$ we have that 
\begin{align*} 
\int_\Omega Q \dive \vv \dx = \int_\Omega Q \dive \jdiv \vv \dx \qquad \text{for all } Q \in \Q(\T_N).
\end{align*}
\item $\jdiv$ is locally defined, i.e., we have that 
\begin{align*}
\mathcal{J}_{\dive}^{N+1} \vv|_{\mathcal{U}^{N+1}(K)}=\jdiv \vv|_{\mathcal{U}^N(K)}
\end{align*}
for all $\vv \in W_0^{1,s}(\Omega)^d$ and all $K \in \T_N$ such that each element $K'$ in $\mathcal{U}^N(K)$ has remained unrefined\footnote{This means that $\{K' \in \T_N: \overline{K}' \cap \overline{K} \neq 0\}=\{K' \in \T_{N+1}: \overline{K}' \cap \overline{K} \neq 0\}$.}. 
\item $\jdiv$ is locally $W^{1,1}$-stable, i.e., there exists a constant $c>0$ independent of $N$, such that
\begin{align} \label{eq:localstab}
\int_K |\jdiv \vv|+h_K |\nabla \jdiv \vv| \dx \leq c \int_{\mathcal{U}^N(K)} |\vv|+h_{\T_N}(\x)|\nabla \vv| \dx
\end{align}
for all $\vv \in W^{1,s}_0(\Omega)$ and all $K \in \T_N$; here $h_K:=|K|^{\nicefrac{1}{d}}$ and $h_{\T_N}(\x)=|K'|^{\nicefrac{1}{d}}$ for $\x \in K' \in \T_N$.
\end{itemize}
\end{assumption}



\begin{assumption} \label{as:jq}
We assume that for each $N \in \mathbb{N}$ there exists a linear projection operator $\jq:L^1(\Omega) \to \Q(\T_N)$ such that $\jq$ is locally $L^1$-stable, i.e., there exists a constant $c>0$ independent of $N$, such that
\begin{align*} 
\int_K |\jq q| \dx \leq c \int_{\mathcal{U}^N(K)} |q| \dx \qquad \text{for all} \ q \in L^1(\Omega) \ \text{and all} \ K \in \T_N.
\end{align*}
\end{assumption}

As shown in~\cite{DiKrSu:2013}, the local stability~\eqref{eq:localstab} implies the global $W^{1,s}$-stability, i.e., for each $s \in [1,\infty)$, there exists $c_s >0$, such that
\begin{align} \label{eq:globalstability}
\norm{\jdiv \vv}_{1,s} \leq c_s \norm{\vv}_{1,s} \qquad \text{for all} \ \vv \in W^{1,s}_0(\Omega)^d.
\end{align}
Similarly, we obtain the global stability of the operator $\jq:L^1(\Omega) \to \Q(\T_N)$. Moreover, thanks to Assumption~\ref{as:jdiv} the following discrete counterpart of the inf-sup condition holds; see~\cite[Lem.~4.1]{Belenki:12}.

\begin{proposition}
For all $s \in (1,\infty)$ there exists a constant $\beta_s>0$, independent of $N \in \mathbb{N}$, such that
\begin{align} \label{eq:discreteinfsup}
\sup_{\Vv \in \V(\T_N) \setminus \{\mathbf{0}\}} \frac{\int_\Omega Q \dive \Vv \dx}{\norm{\Vv}_{1,s}} \geq \beta_s \norm{Q}_{s'} \qquad \text{for all} \ Q \in \Q_0(\T_N), 
\end{align}
where $s'$ denotes, as usual, the H\"{o}lder conjugate of $s$, i.e., $\nicefrac{1}{s}+\nicefrac{1}{s'}=1$.
\end{proposition}

\begin{remark}
Examples of velocity-pressure pairs of finite elements satisfying the Assumptions~\ref{as:jdiv} and~\ref{as:jq}, as well as the nestedness of the discrete spaces, are given by:
\begin{itemize}
\item The lowest order Taylor--Hood element, i.e., the $\mathbb{P}_2-\mathbb{P}_1$ element,
\item The $\mathbb{P}_2-\mathbb{P}_0$ element;
\end{itemize}
we refer to~\cite{Boffi:13}.
\end{remark}

\begin{remark}
In order to work with discretely (rather than pointwise) divergence-free finite element velocity vectors, we will restrict ourselves in this work to the range $r>\frac{2d}{d+1}$; this is crucial for the bound~\eqref{eq:bbounded} below.   
\end{remark}

\subsection{The Galerkin approximation}
 
First of all, in order to inherit the discrete counterpart of the skew-symmetry of the convection term from the continuous case, we define as in~\cite{DiKrSu:2013, KreuzerSuli:2016} the trilinear form
\begin{align*} 
\mathcal{B}[\vv,\wv,\mathbf{h}]:=\frac{1}{2} \int_\Omega(\vv \otimes \mathbf{h}):\nabla \wv-(\vv\otimes \wv):\nabla \mathbf{h} \dx
\end{align*}
for all $(\vv,\wv,\mathbf{h}) \in W^{1,r}(\Omega)^d \times W^{1,r}(\Omega)^d\times W^{1,\tilde{r}'}(\Omega)^d$. Then, it can be shown that the trilinear form is bounded in the sense that 
\begin{align} \label{eq:bbounded}
\B[\vv,\wv,\mathbf{h}] \leq \Cb \norm{\vv}_{1,r} \norm{\wv}_{1,r} \norm{\mathbf{h}}_{1,\tilde{r}'}
\end{align}
for some constant $\Cb>0$; see, e.g.,~\cite[\S 3.3]{KreuzerSuli:2016}. Moreover, since $\tilde{r}' \geq r$, the following skew symmetry property holds:
\begin{align} \label{eq:skewsym}
\mathcal{B}[\uv;\vv,\vv]=0 \qquad \text{for all} \ \uv \in W^{1,r}(\Omega)^d, \ \vv \in W^{1,\tilde{r}'}(\Omega)^d. 
\end{align}
Then, the regularised discrete problem is to find $(\Uv_N^n,P_N^n) \in \V(\T_N) \times \Q_0(\T_N)$ such that
\begin{subequations} \label{eq:discreteproblem}
\begin{align}
\int_\Omega \Sv^n(\D \Uv_N^n): \D\Vv \dx + \B[\Uv_N^n;\Uv_N^n,\Vv]-\int_\Omega P_N^n \dive \Vv \dx&= \int_\Omega \fv \cdot \Vv \dx, \label{eq:dppde}\\
\int_\Omega Q \dive \Uv_N^n \dx&=0, \label{eq:dpic}
\end{align}
\end{subequations}
for all $\Vv \in \V(\T_N)$ and $Q \in \Q(\T_N)$. We note that~\eqref{eq:discreteproblem} has a solution, see~\cite[p.~1344]{KreuzerSuli:2016}.

\begin{remark} \label{rem:convectionpoly}
In the (semi-)discrete setting, the convection term can be reformulated as
\begin{align*}
\B[\Uv;\Uv,\wv]=\int_\Omega \mathrm{B}[\Uv,\Uv] \cdot \wv \dx, \qquad \Uv \in \V(\T_N), \ \wv \in W_0^{1,\tilde{r}'}(\Omega)^d,
\end{align*}
where $\mathrm{B}[\Uv,\Uv] \in \mathbb{P}_{2i^\star-1}(\T_N)$, which is the space of piecewise $\mathbb{P}_{2i^\star-1}$ functions on $\mathcal{T}_N$.
\end{remark}

\subsubsection{Linear approximation}
Assume that we have at our disposal an iterative solver for the discrete problem~\eqref{eq:discreteproblem}. For given $N,n \in \mathbb{N}$, let $\Phi_N^n:\V(\T_N) \times \Q_0(\T_N) \to \V(\T_N) \times \Q_0(\T_N)$ be the iteration function of one step of the iterative solver; i.e., 
\begin{align} \label{eq:discreteiteration}
(\Uv_N^{n,\ell+1},P_N^{n,\ell+1}):=\Phi_N^n(\Uv_N^{n,\ell},P_N^{n,\ell}),
\end{align}
where $(\Uv_N^{n,0},P_N^{n,0}) \in \V(\T_N) \times \Q_0(\T_N)$ is an initial guess. Of course, we need to impose some kind of convergence property on the iteration scheme~\eqref{eq:discreteiteration}. To that end, we define the \emph{discrete} residuals $\F_{N,pde}^n:\V(\T_N) \times \Q_0(\T_N) \to \V(\T_N)^\star$ and $\F_{N,ic}^n:\V(\T_N) \to \Q(\T_N)^\star$ by 
\begin{align} \label{eq:discreteresidualpde}
\dprod{\F_{N,pde}^n(\Uv,P),\Vv}:=\int_\Omega \Sv^n(\D \Uv):\D \Vv \dx + \mathcal{B}[\Uv,\Uv,\Vv]-\int_\Omega P \dive \Vv \dx - \int_\Omega \fv \cdot \Vv \dx, 
\end{align}
and
\begin{align} \label{eq:discreteresidualic}
\dprod{\F_{N,ic}^n(\Uv),Q}:=\int_\Omega Q \dive \Uv \dx, 
\end{align}
respectively. 

\begin{assumption} \label{as:discreteresidual}
For any $n,N \in \mathbb{N}$ we have that
\begin{align} \label{eq:discretedualnorm}
\norm{\F_{N,pde}^{n}(\uvnl,\pnl)}_{N,-1,r'} :=\sup_{\Vv \in \V(\T_N) \setminus\{\mathbf{0}\}} \frac{\dprod{\F_{N,pde}^{n}(\uvnl,\pnl),\Vv}}{\norm{\Vv}_{1,r}} \to 0 \qquad \text{as} \ \ell \to \infty,
\end{align}
and 
\begin{align*} 
\norm{\F^n_{N,ic}(\uvnl)}_{N,-1,r}:=\sup_{Q \in \Q(\T_N) \setminus\{0\}} \frac{\dprod{\F^n_{N,ic}(\uvnl),Q}}{\norm{Q}_{r'}} \to 0 \qquad \text{as} \ \ell \to \infty.
\end{align*}
\end{assumption}

\subsection{A posteriori error analysis}

In this subsection, we will recall an a posteriori error estimate from~\cite[\S 4]{KreuzerSuli:2016} in a slightly modified form: As has been already pointed out in the introduction, in contrast to~\cite{KreuzerSuli:2016}, we take into account inexact finite element approximations obtained by the iteration scheme~\eqref{eq:discreteiteration}. We further note that the proof of the a posteriori error estimate only requires some minor modifications in the analysis from~\cite[\S 4]{KreuzerSuli:2016}, and thus the details will be omitted. \\

Before we can state any results, we need some preparatory work. First, let us define the residual by 
\begin{align*}
\R(\uv,p,\Sv):=(\R^{pde}(\uv,p,\Sv),\R^{ic}(\uv)) \in W^{-1,\tilde{r}}(\Omega)^d \times L_0^{r}(\Omega),
\end{align*}
where, for $(\uv,p,\Sv) \in W_0^{1,r}(\Omega)^d \times L_0^{\tilde{r}}(\Omega) \times L^{r'}(\Omega;\Rsym)$,
\begin{align}
\dprod{\R^{pde}(\uv,p,\Sv),\vv}&:=\int_\Omega \Sv:\D \vv \dx + \mathcal{B}[\uv;\uv,\vv]-\int_\Omega p \dive \vv \dx - \int_\Omega \fv \cdot \vv \dx, \label{eq:rpde} \\ 
\dprod{\R^{ic}(\uv),q}&:=\int_\Omega q \dive \uv \dx \nonumber 
\end{align}
for any $\vv \in W_0^{1,\tilde{r}'}(\Omega)^d$ and $q \in L_0^{r'}(\Omega)$. Moreover, we adopt the following assumption concerning the existence of a quasi-interpolation operator.
\begin{assumption}
We assume that for any $N \in \mathbb{N}$ there exists a linear projection operator $\Pi_N:L^1(\Omega;\Rsym) \to \mathbb{P}_{i^\star-1}(\T_N;\Rsym)$, such that $\Pi_N$ is locally $L^1$-stable; i.e., there exists a constant $c>0$, independent of $N$, such that
\[
\int_K |\Pi_N \Sv| \dx \leq c \int_{\mathcal{U}^N(K)} |\Sv| \dx \qquad \text{for all} \ \Sv \in L^1(\Omega;\Rsym) \ \text{and all} \ K \in \T_N.
\]
\end{assumption}

\begin{remark}
For instance, we may define $\Pi_N:L^1(\Omega;\Rsym) \to \mathbb{P}_{i^\star-1}(\T_N;\Rsym)$ by 
\begin{align*}
\int_\Omega \Pi_N \Sv:\mathbf{D} \dx=\int_\Omega \Sv:\mathbf{D} \dx \qquad \text{for all} \ \mathbf{D} \in \mathbb{P}_{i^\star-1}(\T_N;\Rsym).
\end{align*}
\end{remark}

For the rest of the paper let $t$ and $\tilde{t}$ be such that
\begin{subequations} \label{eq:tdef}
\begin{align}
\frac{2d}{d+1}&<t<r &&\text{and}&&\tilde{t}:=\frac{1}{2}\frac{td}{d-t}, &&  \text{if} \ r \leq \frac{3d}{d+2},\\
t&=r && \text{and}&& \tilde{t}=t'=\tilde{r}=r', && \text{otherwise}.
\end{align}
\end{subequations}
We note that $\tilde{t}<\tilde{r}$ for $r \leq \frac{3d}{d+2}$. Then, for any $K \in \T_N$ and $(\Uv,P,\Sv) \in \V(\T_N) \times \Q_0(\T_N) \times L^{r'}(\Omega;\Rsym)$, we define the local error indicators as follows:
\begin{subequations}
\begin{align*}
\mathcal{E}_N^{pde}(\Uv,P,\Sv;K)&:=\norm{h_K (-\dive \Pi_N \Sv+\mathrm{B}[\Uv,\Uv]+ \nabla P-\fv)}_{\tilde{t},K}^{\tilde{t}} \\
& \qquad +\norm{h_{\partial K}^{\nicefrac{1}{\tilde{t}}} \jmp{\Pi_N \Sv-P \mathbf{I}}}_{\tilde{t},\partial{K}}^{\tilde{t}}+\norm{\Sv-\Pi_N \Sv}_{\tilde{t},K}^{\tilde{t}}, \\
\mathcal{E}_N^{ic}(\Uv;K)&:=\norm{\dive \Uv}_{t,K}^t, \\
\mathcal{E}_N(\Uv,P,\Sv;K)&:=\mathcal{E}_N^{pde}(\Uv,P,\Sv;K)+\mathcal{E}_N^{ic}(\Uv;K), 
\end{align*}
\end{subequations}
where $h_K:=|K|^{\nicefrac{1}{d}}$ and $h_{\partial K}:=|\partial K|^{\nicefrac{1}{d-1}}$. For $\mathcal{M} \subseteq \T_N$ we further define
\begin{align*}
\mathcal{E}_N^{pde}(\Uv,P,\Sv;\mathcal{M})&:=\sum_{K \in \mathcal{M}}\mathcal{E}_N^{pde}(\Uv,P,\Sv;K),\\
\mathcal{E}_N^{ic}(\Uv,\mathcal{M})&:=\norm{\dive \Uv}_{t; \Omega(\mathcal{M})}^t,
\end{align*}
where $\Omega(\mathcal{M})=\bigcup_{K \in \mathcal{M}} K$, and
\begin{align*}
\mathcal{E}_N^{pde}(\Uv,P,\Sv)&:=\mathcal{E}_N^{pde}(\Uv,P,\Sv;\T_N), \\
\mathcal{E}_N^{ic}(\Uv)&:=\mathcal{E}_N^{ic}(\Uv;\T_N)=\norm{\dive \Uv}_{t}^t.
\end{align*}
Finally, we set
\begin{align*}
\mathcal{E}_N(\Uv,P,\Sv):=\mathcal{E}_N^{pde}(\Uv,P,\Sv)+\mathcal{E}_{N}^{ic}(\Uv).
\end{align*}
Next, we introduce a graph approximation error as follows: for $\mathbf{D} \in L^r(\Omega;\Rsym)$ and $\Sv \in L^{r'}(\Omega;\Rsym)$ let
\begin{align} \label{eq:graphapproxerror}
\mathcal{E}_{\mathcal{A}}(\mathbf{D},\Sv):=\int_\Omega \inf_{(\ukappa,\udelta) \in \mathcal{A}(\x)} |\mathbf{D}-\ukappa|^r+|\Sv-\udelta|^{r'} \dx.
\end{align}
Now we are in a position to recall two important results from~\cite{KreuzerSuli:2016}. The second result is slightly modified on account of inexact finite element approximations of~\eqref{eq:discreteproblem}, meaning that we do not require $(\Uv,P)=(\Uv_N^n,P_N^n)$ in Theorem~\ref{thm:upperbound} below, in contrast to the corresponding results in~\cite{KreuzerSuli:2016}.  
 
\begin{lemma}[{\hspace{1sp}\cite[Lem.~4.2]{KreuzerSuli:2016}}] \label{lem:solution}
The triple $(\uv,p,\Sv) \in W_0^{1,r}(\Omega)^d \times L_0^{\tilde{r}}(\Omega) \times L^{r'}(\Omega;\Rsym)$ is a solution of~\eqref{eq:pde} if and only if 
\begin{align} \label{eq:equivprop}
\R(\uv,p,\Sv)=0 \quad \text{in} \ W^{-1,\tilde{t}}(\Omega)^d \times L_0^t(\Omega) \quad \text{and} \quad \mathcal{E}_\A(\D \uv,\Sv)=0.
\end{align}
\end{lemma}


\begin{theorem}[{\hspace{1sp}\cite[Thm.~4.3 \& Cor.~4.4]{KreuzerSuli:2016}}] \label{thm:upperbound}
For any $\Uv \in \V(\T_N)$ and $P \in \Q_0(\T_N)$ the following bounds hold true: 
\begin{equation} \label{eq:pdelocalbound}
\begin{aligned}
\dprod{\R^{pde}(\Uv,P,\Sv^n(\D \Uv)),\vv} &\leq C_1 \sum_{K \in \mathcal{T}_N} \mathcal{E}_N^{pde}(\Uv,P,\Sv^n(\D \Uv);K)^{\nicefrac{1}{\tilde{t}}} \norm{\vv}_{1,\tilde{t}',\mathcal{U}^N(K)} \\
& \quad + C_2 \norm{\F_{N,pde}^{n}(\Uv,P)}_{N,-1,\tilde{t}} \norm{\jdiv \vv}_{1,\tilde{t}'}
\end{aligned}
\end{equation} 
for all $\vv \in W_0^{1,\tilde{t}'}(\Omega)$, and
\begin{align} \label{eq:pdeglobalbound}
\norm{\R^{pde}(\Uv,P,\Sv^n(\D \Uv))}_{W^{-1,\tilde{t}}(\Omega)} \leq \tilde{C}_1 \mathcal{E}_N^{pde}(\Uv,P,\Sv^n(\D \Uv))^{\nicefrac{1}{\tilde{t}}}+\tilde{C}_2 \norm{\F_{N,pde}^{n}(\Uv,P)}_{N,-1,\tilde{t}},
\end{align}
where $C_1,\tilde{C}_1,C_2$, and $\tilde{C}_2$ depend only on the shape-regularity, $\tilde{t}$, and on the dimension $d$, however are independent of $n,N \in \mathbb{N}$. Moreover, 
\begin{align} \label{eq:iclocalbound}
\dprod{\R^{ic}(\Uv),Q} \leq \sum_{K \in \T_N} \mathcal{E}_{N}^{ic}(\Uv;K)^{\nicefrac{1}{t}}\norm{Q}_{t',K}
\end{align}
for any $Q \in L^{t'}(\Omega)$, and
\begin{align} \label{eq:icglobalbound}
\norm{\R^{ic}(\Uv)}_{t}^t \leq \mathcal{E}_N^{ic}(\Uv).
\end{align}
\end{theorem}

\begin{proof}
By the definition of $\R^{pde}$, cf.~\eqref{eq:rpde}, and Remark~\ref{rem:convectionpoly} we have that
\begin{align*}
\dprod{\R^{pde}(\Uv,P,\Sv^n(\D \Uv)),\vv}=\int_\Omega \Sv^n(\D \Uv):\D \vv+\mathrm{B}[\Uv,\Uv]\cdot \vv-P\dive \vv -\fv\cdot \vv \dx
\end{align*}
for all $\vv \in W_0^{1,\tilde{t}'}$. Then, since $\jdiv:W_0^{1,\tilde{t}'}(\Omega)^d \to \V(\T_N)$, it follows from~\eqref{eq:discreteresidualpde} that
\begin{align*}
\dprod{\R^{pde}(\Uv,P,\Sv^n(\D \Uv),\vv}&=\int_\Omega \Pi_N \Sv^n(\D \Uv):\D(\vv-\jdiv \vv)+\mathrm{B}[\Uv,\Uv] \cdot (\vv -\jdiv \vv) \\
& \quad -P\dive(\vv-\jdiv \vv) \dx -\int_\Omega\fv \cdot (\vv-\jdiv \vv)\dx \\ 
&\quad + \int_\Omega (\Sv^n(\D \Uv)-\Pi_N\Sv^n(\D \Uv)):\D (\vv-\jdiv \vv) \dx \\
& \quad + \dprod{\F_{N,pde}^n(\Uv,P),\jdiv \vv}.
\end{align*}
We note that, compared to the proof of~\cite[Thm.~4.3]{KreuzerSuli:2016}, we simply had to add the last term $\dprod{\F_{N,pde}^n(\Uv,P),\jdiv \vv}$ on account of general discrete elements $\Uv \in \V(\T_N)$ and $P \in \Q_0(\T_N)$; concerning this term, \eqref{eq:discretedualnorm} and~\eqref{eq:globalstability} imply that
\begin{align*}
\dprod{\F_{N,pde}^n(\Uv,P),\jdiv \vv} &\leq \norm{\F_{N,pde}^n(\Uv,P)}_{N,-1,\tilde{t}}\norm{\jdiv \vv}_{1,\tilde{t}'} \leq c_{\tilde{t}'} \norm{\F_{N,pde}^n(\Uv,P)}_{N,-1,\tilde{t}}\norm{\vv}_{1,\tilde{t}'}.
\end{align*}
The remainder can be bounded by standard techniques in a posteriori error analysis; indeed, similarly as was shown in the proof of~\cite[Thm.~4.3]{KreuzerSuli:2016}, we have by a local application of the divergence theorem, H\"{o}lder's inequality, the local stability of $\jdiv$, cf.~\eqref{eq:localstab}, an interpolation estimate on $\V(\T_N)$, and a scaled trace theorem that
\begin{multline*}
\dprod{\R^{pde}(\Uv,P,\Sv^n(\D \Uv)),\vv}  \\
\begin{aligned}
&\leq C \Big(\sum_{K \in \T_N} \Big\{\norm{h_K (-\dive \Pi_N\Sv^n(\D \Uv) +\mathrm{B}[\Uv,\Uv]+\nabla P-\fv)}_{\tilde{t},K} \norm{\vv}_{1,\tilde{t}',\mathcal{U}^N(K)}\\
& \quad +\norm{h_K^{1/\tilde{t}} \jmp{\Pi_N\Sv^n(\D \Uv)-P\mathbf{I}}}_{\tilde{t},\partial K} \norm{\vv}_{1,\tilde{t}',\mathcal{U}^N(K)} \\
& \quad +\norm{\Sv^n(\D \Uv)-\Pi_N \Sv^n(\D \Uv)}_{\tilde{t},K}\norm{\vv}_{1,\tilde{t}',K}\Big\}\\
& \quad + \norm{\F_{N,pde}^n(\Uv,P))}_{N,-1,\tilde{t}}\norm{\jdiv \vv}_{1,\tilde{t}'}\Big)
\end{aligned} 
\end{multline*}
for some constant $C>0$ independent of $n,N \in \mathbb{N}$;
both~\eqref{eq:pdelocalbound} and~\eqref{eq:pdeglobalbound} can be derived from this inequality by further applications of H\"{o}lder's inequality and by the shape-regularity of the mesh. The remainder, i.e.,~\eqref{eq:iclocalbound} and~\eqref{eq:icglobalbound}, follow immediately from the definitions of $\R^{ic}$ and $\mathcal{E}^{ic}$, and H\"{o}lder's inequality.
\end{proof}


\section{Convergent adaptive iterative linearised finite element method} \label{sec:ailfem}

In the following, we will present an adaptive algorithm which exploits an interplay of the iterative linear method~\eqref{eq:discreteiteration}, adaptive mesh refinement, and graph approximation. In particular, we modify the AFEM and the corresponding analysis from~\cite{KreuzerSuli:2016} by taking into account inexact finite element approximations of~\eqref{eq:discreteproblem} obtained by the iteration scheme~\eqref{eq:discreteiteration}. We note that this can be done in a rather straightforward manner, and that the proofs only require minor adaptions; therefore, we will only give rough sketches of the proofs and highlight where the main differences occur. \\

For the remainder of this work, we assume that there exists a \emph{computable} bound $\eta_{\A,n}:L^r(\Omega;\Rsym)\to \mathbb{R}_{\geq 0}$ such that
\begin{align*}
\mathcal{E}_\A(\mathbf{D},\Sv^n(\mathbf{D})) \leq \eta_{\A,n}(\mathbf{D}) \qquad \text{for all} \ \mathbf{D} \in L^r(\Omega;\Rsym).   
\end{align*}
For instance, we may define 
\begin{align*}
\eta_{\A,n}(\mathbf{D}):=\int_\Omega |\Sv^n(\mathbf{D})-\Sv^\star(\mathbf{D})|^{r'} \dx.
\end{align*}

For the purpose of the convergence of the adaptive iterative linearised finite element method (AILFEM), cf.~Algorithm~\ref{alg:AILFEM}, we need to further impose the following assumption.

\begin{assumption} \label{as:graphapp2}
For every $\epsilon>0$ there exists a positive integer $\widehat{n} \in \mathbb{N}$ such that
\[\eta_{\A,n}(\D \vv)<\epsilon \qquad \text{for all} \ \vv \in W_0^{1,r}(\Omega)^d \ \text{and} \ n > \widehat{n}.\]
\end{assumption}
We note that this hypothesis is stronger than the related one~\cite[Assumption~5.6]{KreuzerSuli:2016}. In particular, in contrast to~\cite{KreuzerSuli:2016}, we do \emph{not} want to include the term $\mathcal{E}_\A(\mathbf{D},\Sv^n(\mathbf{D}))$ in our algorithm below, since it is, in general, not computable; in this case, however, \cite[Assumption~5.6]{KreuzerSuli:2016} is not sufficient to prove Step 4 of the proof of Theorem~\ref{thm:main} in Appendix~\ref{app:1}.\\

In Algorithm~\ref{alg:AILFEM} and below we use the following abbreviations:
\begin{align*} 
\F_{N,pde}^{n,\ell}&:=\F_{N,pde}^n(\Uv_N^{n,\ell},P_{N}^{n,\ell}), \\ 
\F_{N,ic}^{n,\ell}&:=\F_{N,ic}^n(\uvnl), \\
\mathcal{E}_N(n,\ell)&:=\mathcal{E}_N(\Uv_N^{n,\ell},P_{N}^{n,\ell}),\\
\mathcal{E}_\A(N,n,\ell)&:=\mathcal{E}_\A(\D \Uv_N^{n,\ell},\Sv^n(\D \Uv_N^{n,\ell})), \nonumber \\ 
\eta_{\A}(N,n,\ell)&:=\eta_{\A,n}(\D \Uv_N^{n,\ell}). \nonumber
\end{align*}
Moreover, $\zeta:\mathbb{N} \to \mathbb{R}_{>0}$ is any given function such that 
\begin{align} \label{eq:zetafunction}
\zeta(N) \to 0 \qquad \text{as} \ N \to \infty.
\end{align}

\begin{algorithm} 
\caption{Adaptive iterative linearised finite element method (AILFEM)}
\label{alg:AILFEM}
\begin{algorithmic}[1]
\State Input initial guesses $\Uv_0^{1,0} \in \V(\T_0)$ and $P_0^{1,0} \in \Q_0(\T_0)$ for an initial triangulation $\T_0$ of $\Omega$.
\State Let $N=\ell=0$ and $n_0=1$.
\Loop
\While {$\norm{\F_{N,pde}^{n_N,\ell}}_{N,-1,t'}+\norm{\F_{N,ic}^{n_N,\ell}}_{N,-1,r}\geq \inf\{\max\{\mathcal{E}_N(n_N,\ell),\eta_{\A}(N,n_N,\ell)\},\zeta(N)\}$}
\State Let $(\Uv_N^{n_N,\ell+1},P_N^{n_N,\ell+1})=\Phi_N^n(\Uv_N^{n_N,\ell},P_N^{n_N,\ell})$.
\State $\ell \leftarrow \ell+1$.
\EndWhile
\State Set $\ell_N:=\ell$.
\If {$\mathcal{E}_N(n_N,\ell_N) \geq \eta_{\A}(N,n_N,\ell_N)$}
\State MARK and REFINE the mesh to obtain $\T_{N+1}$ from $\T_{N}$.
\State Let $n_{N+1}=n_N$.
\Else
\State Let $n_{N+1}=n_N+1$.
\EndIf
\State $N \leftarrow N+1$.
\EndLoop
\end{algorithmic}
\end{algorithm}

\begin{remark} \label{rem:dualnormequiv}
Since $t \leq \max\{\tilde{t}',r\}$ and as $\Omega$ is a bounded domain, there exist constants $C_r,C_t >0$ independent of $N,n,\ell$ such that
\[
\norm{\F_{N,pde}^{n,\ell}}_{N,-1,r'} \leq C_r \norm{\F_{N,pde}^{n,\ell}}_{N,-1,t'} \qquad \text{and} \qquad \norm{\F_{N,pde}^{n,\ell}}_{N,-1,\tilde{t}} \leq C_t \norm{\F_{N,pde}^{n,\ell}}_{N,-1,t'}. 
\]
Similarly, 
\[
\norm{\F_{N,ic}^{n,\ell}}_{N,-1,t} \leq C \norm{\F_{N,ic}^{n,\ell}}_{N,-1,r}
\] 
for some constant $C>0$ independent of $N,n,\ell$. Moreover, in view of Assumption~\ref{as:discreteresidual} and the equivalence of the norms in finite-dimensional spaces, the while-loop in Algorithm~\ref{alg:AILFEM} (lines 4--7) terminates after finitely many steps.
\end{remark}

For the subroutine MARK we may either employ the Maximum or D\"orfler strategy, see~\cite{Babuska:1984} and ~\cite{Doerfler:96}, respectively. Subsequently, we use a refinement procedure, which bisects all marked elements at least once, removes any hanging nodes, and guarantees the shape-regularity of the sequence of meshes $\{\mathcal{T}_N\}_N$. For instance, but not exclusively, we can apply Newest Vertex Bisection, cf.~\cite{Mitchell:91}. \\

We need one more assumption for the convergence of the sequence of iterates generated:

\begin{assumption} \label{ass:bounded}
Let $\{(\Uv_N^{n_N,\ell_N},P_N^{n_N,\ell_N})\}_{N} \subset W_0^r(\Omega)^d \times L_0^{\tilde{r}}(\Omega)$ be the sequence generated by Algorithm~\ref{alg:AILFEM}. Then, at least one of the following two boundedness properties holds:
\begin{itemize}
\item There exist constants $c_p,\tilde{c}_p$ and $s \in [1,r)$ such that
\begin{align} \label{eq:assufpAILFEM}
\dprod{\F_{N,ic}^{n_N,\ell_N},P_N^{n_N,\ell_N}} \leq c_{p}+\tilde{c}_p \norm{\Uv_N^{n_N,\ell_N}}_{1,r}^{s} \qquad \text{for all} \ N=0,1,2,\dotsc;
\end{align}
\item There exists a constant $c_U>0$ such that
\begin{align} \label{eq:asubounded}
\norm{\Uv^{n_N}_N-\Uv_N^{n_N,\ell_N}}_{1,r} \leq c_U \qquad \text{for all} \ N=0,1,2,\dotsc,
\end{align}
where $\Uv^{n_N}_N$ is the velocity vector of a solution of~\eqref{eq:discreteproblem}. 
\end{itemize}
\end{assumption}

\begin{remark} \label{rem:graphas}
We emphasize that~\eqref{eq:assufpAILFEM} is certainly satisfied if the sequence $\{P_N^{n_N,\ell_N}\}_{N}$ is uniformly bounded in $L^{r'}(\Omega)$, or if $\Uv_N^{n_N,\ell_N}$, for $N=0,1,2,\dotsc$, are discretely divergence-free; in the latter case, we indeed have that $\dprod{\F_{N,ic}^{n_N,\ell_N},P_N^{n_N,\ell_N}}=0$.
\end{remark}

The purpose of Assumption~\ref{ass:bounded} is to guarantee the uniform boundedness of $\norm{\Uv_N^{n_N,\ell_N}}_{1,r}$, and, in turn by~\eqref{eq:boundedS}, the uniform boundedness of $\norm{\Sv^{n_N}(\D\Uv_N^{n_N,\ell_N})}_{r'}$.

\begin{lemma} \label{lem:uniformbound}
Given Assumption~\ref{ass:bounded}, we have that 
\begin{align} \label{eq:uniformbound}
\norm{\Uv_N^{n_N,\ell_N}}_{1,r}+\norm{\Sv^{n_N}(\D\Uv_N^{n_N,\ell_N})}_{r'} \leq c_\fv \qquad \text{for all} \ N=0,1,2,\dotsc,
\end{align}
where $c_\fv$ depends on the source function $\fv$, the constants $\tilde{c}_1, \, \tilde{c}_2$ from Assumption~\ref{as:graphapp}, the functions $\tilde{k}$ and $\tilde{m}$ from Assumption~\ref{as:graphapp}, the constants $c_p$ and $\tilde{c}_p$ from Assumption~\ref{ass:bounded}, $\zeta_{\max}:=\max_{N \in \mathbb{N}} \zeta(N)$, and the constant $C_r$ from Remark~\ref{rem:dualnormequiv}. 
\end{lemma}

\begin{proof}
By our observation from before, we only need to verify the uniform boundedness of the generated sequence of velocity vectors. The assumption~\eqref{eq:asubounded} immediately implies the uniform boundedness of $\norm{\Uv_N^{n_N,\ell_N}}_{1,r}$, since $\norm{\uvn}_{1,r}$ is uniformly bounded in $n$ and $N$, cf.~\cite[\S 3.4]{KreuzerSuli:2016}.

Now let~\eqref{eq:assufpAILFEM} be given. For the proof of this case we will proceed along the lines of~\cite[\S 3.4]{KreuzerSuli:2016}. By the definitions~\eqref{eq:discreteresidualpde} and \eqref{eq:discreteresidualic} and the skew-symmetry~\eqref{eq:skewsym} we have that
\begin{align*}
\int_\Omega \Sv^{n_N}(\D \Uv_N^{n_N,\ell_N}):\D \Uv_N^{n_N,\ell_N} \dx= \int_\Omega \fv \cdot \Uv_N^{n_N,\ell_N} \dx+ \dprod{\F_{N,ic}^{n_N,\ell_N},P_N^{n_N,\ell_N}}+\dprod{\F_{N,pde}^{n_N,\ell_N},\Uv_N^{n_N,\ell_N}};
\end{align*} 
we note the latter two summands arise, in contrast to~\cite[\S 3.4]{KreuzerSuli:2016}, since we are considering inexact finite element approximations of~\eqref{eq:discreteproblem} that are possibly not discretely divergence-free. Invoking the assumption~\eqref{eq:assufpAILFEM} and the modus operandi of Algorithm~\ref{alg:AILFEM}, we obtain the upper bound
\begin{align*}
\int_\Omega \Sv^{n_N}(\D \Uv_N^{n_N,\ell_N}):\D \Uv_N^{n_N,\ell_N} \dx \leq (\norm{\fv}_{-1,r'}+C_r\zeta_{\max})\norm{\Uv_N^{n_N,\ell_N}}_{1,r}+c_p+\tilde{c}_p \norm{\Uv_N^{n_N,\ell_N}}_{1,r}^s.
\end{align*}
Consequently, since $1 \leq s <r$, Assumption~\ref{as:graphapp}, cf.~\eqref{eq:monotoneS}, implies that $\norm{\Uv_N^{n_N,\ell_N}}_{1,r}$ is uniformly bounded, and thus the claim is verified. 
\end{proof}

Finally, we will state the convergence property of the AILFEM, cf.~Algorithm~\ref{alg:AILFEM}.
\begin{theorem} [{\hspace{1sp}\cite[Cor.~5.8]{KreuzerSuli:2016}}] \label{thm:main}
Let $\{(\Uv_N^{n_N,\ell_N},P_N^{n_N,\ell_N})\}_{N} \subset W_0^r(\Omega)^d \times L_0^{\tilde{r}}(\Omega)$ be the sequence generated by Algorithm~\ref{alg:AILFEM}. Given all the assumptions made so far, we have that there exists a subsequence of $\{(\Uv_N^{n_N,\ell_N},P_N^{n_N,\ell_N},\Sv^{n_N}(\D \Uv_N^{n_N,\ell_N}))\}_{N}$ with weak limit $(\uv_\infty,p_\infty,\Sv_\infty) \in W_0^{1,r}(\Omega)^d \times L_0^{\tilde{r}}(\Omega) \times L^{r'}(\Omega; \Rsym)$ that solves~\eqref{eq:pde}. 
\end{theorem} 

A sketch of the proof of Theorem~\ref{thm:main} shall be postponed to Appendix~\ref{app:1}.

\section{Application to Bingham fluids}

In this section we will apply AILFEM, cf.~Algorithm~\ref{alg:AILFEM}, to Bingham fluids\footnote{We note that this approach can equally be applied to certain Herschel-Bulkley fluids. However, for simplicity of the presentation, we will 
restrict the analysis to Bingham fluids.}. Those fluids have wide-ranging applications, for instance in gas and oil industry, cf.~\cite{Frigaard:17}, or in the food industry, cf.~\cite{Ortega:19}, and were introduced in the early 1920s by E.C.~Bingham, cf.~\cite{bingham:1922}. Bingham fluids behave like a rigid body as long as the shear stress is below a threshold value, the so-called yield stress $\sigma$, and like a Newtonian fluid for shear stress exceeding this value. In particular, one has that
\begin{align*}
|\Dv|=0 \quad \Leftrightarrow \quad &|\Sv| \leq \sigma, \\
|\Dv| > 0 \quad \Leftrightarrow  \quad &\Sv=\sigma\frac{\Dv}{|\Dv|}+2 \nu \Dv, 
\end{align*}  
where $\nu$ denotes the viscosity of the fluid. In turn, a measurable selection of the corresponding maximal monotone $2$-graph is given by
\begin{align*} 
\Sv^\star(\Dv)=\begin{cases} 0 & \text{if} \ |\Dv|=0, \\
\sigma \frac{\Dv}{|\Dv|}+2 \nu \Dv & \text{else}.
\end{cases}
\end{align*}
For the approximation of $\Sv^\star:\Omega \times \Rsym \to \Rsym$ by a sequence of strictly monotone mappings we will use the Bercovier--Engelman regularisation, cf.~\cite{Bercovier:80}: 
\begin{align} \label{eq:graphapprox}
\Sv^n(\Dv):=\left(\frac{\sigma}{\sqrt{|\Dv|^2+n^{-2}}}+2 \nu \right) \Dv= \left(\frac{\sigma}{|\Dv|_{n}}+2 \nu \right) \Dv, \qquad n=1,2,\dotsc,
\end{align}
where $|\ukappa|_n:=\sqrt{|\ukappa|^2+n^{-2}}$ for  $\ukappa \in \Rsym$. Indeed, this graph approximation satisfies Assumption~\ref{as:graphapp}, and, for $\eta_{\A,n}(\mathbf{D}):=\int_\Omega|\Sv^n(\mathbf{D})-\Sv^\star(\mathbf{D})|^{r'}\dx$, Assumption~\ref{as:graphapp2} is satisfied as well; cf.~\cite[\S 7]{KreuzerSuli:2016}. Then, the weak formulation of the regularised discrete Bingham fluid flow problem reads as follows: find $(\Uv_N^n, P_N^n) \in \V(\T_N) \times \Q_0(\T_N)$ such that
\begin{subequations} \label{eq:regdis}
\begin{align}
\int_\Omega \left(\frac{\sigma}{|\D \Uv_N^n|_n}+2 \nu \right)\D \Uv_N^n:\D \Vv \dx + \B[\Uv_N^n;\Uv_N^n,\Vv]-\int_\Omega P_N^n \dive \Vv\dx &= \int_\Omega \fv \cdot \Vv \dx  \label{eq:regdispde}\\
- \int_\Omega Q \dive \Uv_N^n \dx &= 0 \label{eq:regdisic}
\end{align}
\end{subequations} 
for all $(\Vv,Q) \in \V(\T_N) \times \Q(\T_N)$. Upon defining
\begin{align} \label{eq:andef}
a_n(\Uv;\Vv,\Wv):=\int_\Omega \left(\frac{\sigma}{|\D \Uv|_n}+2 \nu\right) \D \Vv : \D \Wv \dx, \qquad \Uv,\Vv,\Wv \in \V(\T_N),
\end{align}
and
\begin{align*}
b(Q,\Vv):= - \int Q \dive \Vv \dx, \qquad \Vv \in \V(\T_N), \ Q \in \Q(\T_N),
\end{align*}
the problem~\eqref{eq:regdis} can be stated equivalently as follows:  find $(\Uv_N^n, P_N^n) \in \V(\T_N) \times \Q_0(\T_N)$ such that
\begin{subequations} \label{eq:weakbingham}
\begin{align} 
a_n(\Uv_N^n;\Uv_N^n,\Vv)+\B[\Uv_N^n;\Uv_N^n,\Vv]+b(P_N^{n},\Vv)&=\int_\Omega \fv \cdot \Vv \dx \label{eq:weakbinghampde}\\
b(Q,\Uv_N^n)&=0 \label{eq:weakbinghamic}
\end{align}
\end{subequations}
for all $(\Vv,Q) \in \V(\T_N) \times \Q(\T_N)$. 

In the following, we will examine the regularised Bingham problem~\eqref{eq:weakbingham}, and thereby introduce an iterative solver; initially, we will restrict ourselves to the case of a slowly flowing fluid, i.e., we will neglect the convection term, and subsequently analyse the general case involving the convection term.

\subsection{Ka\v{c}anov iteration for slowly flowing Bingham fluids}

As mentioned before, to begin with we will neglect the convection term. Consequently, the problem amounts to finding $(\Uv_N^n,P_N^n) \in \V(\T_N) \times \Q_0(\T_N)$ such that
\begin{align} \label{eq:weakbinghamconfree}
a_n(\Uv_N^n;\Uv_N^n,\Vv)+b(P_N^{n},\Vv)+b(Q,\Uv_N^n)=\int_\Omega \fv \cdot \Vv \dx
\end{align}
for all $(\Vv,Q) \in \V(\T_N) \times \Q(\T_N)$. We will now state some properties, partially borrowed from~\cite{Aposporidis:2011}, of the form $a_n:\V(\T_N) \times \V(\T_N) \times \V(\T_N) \to \mathbb{R}$ defined in~\eqref{eq:andef}, which will be important in the analysis later on. But first, let us recall Korn's inequality:
\begin{align} \label{eq:korn}
\norm{\D \uv}_2 \leq \norm{\uv}_{1,2} \leq \sqrt{2} \norm{\D \uv}_{2} \qquad \text{for all} \ \uv \in W_0^{1,2}(\Omega)^d.
\end{align}

\begin{lemma} \label{lem:anproperties}
\begin{enumerate}[(a)]
\item For $\mu_n(t):=\frac{\sigma}{\sqrt{t+n^{-2}}}+2 \nu$, $t \geq 0$, we have that
\begin{align} \label{eq:mukey}
2 \nu (t-s) \leq \mu_n(t^2)t-\mu_n(s^2)s \leq (\sigma n+2 \nu)(t-s) \qquad \text{for all} \ t \geq s \geq 0.
\end{align}
\item For any $\Uv \in \V(\T_N)$, the bilinear form $a_n(\Uv;\cdot,\cdot)$ is uniformly coercive with
\begin{align*} 
a_n(\Uv;\Vv,\Vv) \geq \nu \norm{\Vv}_{1,2}^2 \qquad \text{for all } \Vv \in \V(\T_N),
\end{align*}
and bounded with
\begin{align} \label{eq:anbounded}
a_n(\Uv;\Vv,\Wv) \leq (\sigma n +2 \nu) \norm{\Vv}_{1,2} \norm{\Wv}_{1,2} \qquad \text{for all} \ \Vv,\Wv \in \V(\T_N).
\end{align}
\item The strong monotonicity property 
\begin{align} \label{eq:anstrongmon}
a_n(\Uv;\Uv,\Uv-\Vv)-a_n(\Vv;\Vv,\Uv-\Vv) \geq \nu \norm{\Uv-\Vv}_{1,2}^2, 
\end{align}
as well as the Lipschitz continuity
\begin{align} \label{eq:anlipschitz}
a_n(\Uv;\Uv,\Wv)-a_n(\Vv;\Vv,\Wv) \leq (\sqrt{3} \sigma n+2 \nu) \norm{\Uv-\Vv}_{1,2} \norm{\Wv}_{1,2}
\end{align}
hold for all $\Uv,\Vv, \Wv \in \V(\T_N)$.
\end{enumerate}
\end{lemma}

\begin{proof}
In order to prove (a), let us define $\xi_n(t):=\mu_n(t^2)t$. Then, by the mean-value theorem we have that
\begin{align*}
\inf_{\tau \geq 0} \xi'_n(\tau)(t-s) \leq \mu_n(t^2)t-\mu_n(s^2)s \leq \sup_{\tau \geq 0} \xi_n'(\tau)(t-s) \qquad \text{for all} \ t \geq s \geq 0.
\end{align*}
Moreover, a straightforward calculation reveals that
\begin{align*}
\xi_n'(\tau)=2\nu+\frac{\sigma n^{-2}}{(\tau^2+n^{-2})^{\nicefrac{3}{2}}},
\end{align*}
and, in turn,
\begin{align*}
2 \nu \leq \xi_n'(\tau) \leq 2 \nu + \sigma n \qquad \text{for all} \ \tau \geq 0;
\end{align*}
this yields (a). The assertion (b) can be shown in a straightforward manner by invoking Korn's inequality~\eqref{eq:korn}. In order to verify (c), we first note that it can be shown that
\begin{align*}
|\D \Uv -\D \Vv|^2 \geq \frac{|\D \Uv|_n-|\D \Vv|_n}{|\D \Vv|_n} \D \Vv:\D(\Uv-\Vv).
\end{align*}
Consequently,
\begin{multline*}
a_n(\Uv;\Uv,\Uv-\Vv)-a_n(\Vv;\Vv,\Uv-\Vv) \\
\begin{aligned}
&=\int_\Omega 2 \nu |\D \Uv- \D \Vv|^2+ \frac{\sigma}{|\D \Uv|_n}\left(|\D \Uv-\D \Vv|^2-\frac{|\D \Uv|_n-|\D \Vv|_n}{|\D \Vv|_n} \D \Vv:\D(\Uv-\Vv)\right) \dx \\
&\geq \int_\Omega 2 \nu | \D \Uv-\D \Vv|^2 \dx \\
&\geq \nu \norm{\Uv-\Vv}_{1,2}^2.
\end{aligned}
\end{multline*}
Finally, the asserted Lipschitz continuity can be established as in the proof of~\cite[Prop.~2.2(b)]{HeidSuli:21} thanks to~\eqref{eq:mukey}.
\end{proof}

With the properties stated in Lemma~\ref{lem:anproperties} it can be shown that the weak convection-free problem~\eqref{eq:weakbinghamconfree} has a unique solution.

\begin{proposition}[{\hspace{1sp}\cite[Prop.~1]{Aposporidis:2011}}]
The problem~\eqref{eq:weakbinghamconfree} has a unique solution~$(\Uv_N^n,P_N^n) \in \V(\T_N) \times \Q_0(\T_N)$ with $\norm{\Uv_N^{n}}_{1,2} \leq \nu^{-1} C_P \norm{\fv}_{2}$, where $C_P$ is the constant from the Poincar\'{e}--Friedrichs inequality.
\end{proposition}

In order to approximate the unique solution $(\Uv_N^n,P_N^n) \in \V(\T_N) \times \Q_0(\T_N)$ of~\eqref{eq:weakbinghamconfree}, we will employ the Ka\v{c}anov iteration, see, e.g.,\cite{GarauMorinZuppa:11, HaJeSh:97,Zeidler:90} and Ka\v{c}anov's original work~\cite{kacanov:1959}, which is defined as follows: given $\Uv_N^{n,\ell} \in \V(\T_N)$ find $(\Uv_N^{n,\ell+1},P_N^{n,\ell+1}) \in \V(\T_N) \times \Q_0(\T_N)$ such that 
\begin{align} \label{eq:picardconfree}
\ a_n(\Uv_N^{n,\ell};\Uv_N^{n,\ell+1},\Vv)+b(P_N^{n,\ell+1},\Vv)+b(Q,\Uv_N^{n,\ell+1})=\int_\Omega \fv \cdot \Vv \dx
\end{align} 
for all $(\Vv,Q) \in \V(\T_N) \times \Q(\T_N)$, where $\Uv_N^{n,0} \in \V(\T_N)$ is an inital guess. Thanks to Lemma~\ref{lem:anproperties}, the Ka\v{c}anov iteration~\eqref{eq:picardconfree} is well-defined, see, e.g.,~\cite[\S 3]{Aposporidis:2011}. Moreover, Assumption~\ref{ass:bounded} is satisfied since the velocity vectors $\Uv_N^{n,\ell}$ are discretely divergence-free; cf.~Remark~\ref{rem:graphas}. It remains to show that the sequence generated by the Ka\v{c}anov iteration converges to the unique solution established in the proposition above, which will be done based on the works~\cite{HeidWihler:19v2, HeidWihler2:19v1, HeidPraetoriusWihler:21, HeidSuli:21}. For that purpose we define the functional $\H_n:\V(\T_N) \to \mathbb{R}$ by
\begin{align} \label{eq:hn1}
\H_n(\Uv):=\int_\Omega \varphi_n(|\D \Uv|^2)dx-\int_\Omega \fv \cdot \Uv \dx,
\end{align} 
where 
\begin{align*}
\varphi_n(s):=\frac{1}{2} \int_0^s \mu_n(t) \, \mathrm{d} t \quad \text{with} \quad \mu_n(t)=\frac{\sigma}{\sqrt{t+n^{-2}}}+2 \nu,
\end{align*}
for $s,t \geq 0$. In particular,
\begin{align*}
\H_n(\Uv)=\int_\Omega (\sigma \sqrt{|\D \Uv|^2+n^{-2}}+\nu |\D \Uv|^2) \dx-\int_\Omega \fv \cdot \Uv \dx.
\end{align*}
Moreover, we have that
\begin{align*}
\dprod{\H_n'(\Uv),\Vv}=a_n(\Uv;\Uv,\Vv)-\int_\Omega \fv \cdot \Vv \dx,
\end{align*}
where $\H_n'$ denotes the Gateaux derivative of $\H_n$. We will now state some known properties -- in a slightly different setting -- of the functional $\H_n$, cf.~\eqref{eq:hn1}.

\begin{lemma}
\begin{enumerate}[(A)]
\item For any $\Uv,\Vv \in \V(\T_N)$ we have that
\begin{align*}
\H_n(\Uv)-\H_n(\Vv) \geq \frac{1}{2}a_n(\Uv;\Uv,\Uv)-\frac{1}{2}a_n(\Uv;\Vv,\Vv)-\int_\Omega \fv \cdot \Uv \dx +\int_\Omega \fv \cdot \Vv \dx.
\end{align*}
\item Let $\Uv_N^n \in \V_0(\T_N)$ denote the velocity vector of the unique solution of~\eqref{eq:weakbinghamconfree}. Then, $\Uv_N^n$ is the global minimiser of $\H_n$ in $\V_0(\T_N)$, and we have that
\begin{align} \label{eq:equiv}
\frac{\nu}{2}\norm{\Uv-\Uv_N^n}_{1,2}^2 \leq \H_n(\Uv)-\H_n(\Uv_N^n) \leq \frac{\sqrt{3}\sigma n+2 \nu}{2}\norm{\Uv-\Uv_N^n}_{1,2}^2
\end{align}
for all $\Uv \in \V_0(\T_N)$.
\item Let $\{\uvnl\}_\ell \subset \V_0(\T_N)$ be the sequence generated by~\eqref{eq:picardconfree}\footnote{Without loss of generality we assume here and in the following that $\Uv_N^{n,0} \in \V_0(\T_N)$.}. Then, the following monotonicity property holds:
\begin{align} \label{eq:hnmono}
\H_n(\uvnl)-\H_n(\uvnlo) \geq \frac{\nu}{2} \norm{\uvnlo-\uvnl}_{1,2}^2 \qquad \text{for all} \ \ell=0,1,2,\dotsc.
\end{align}
\end{enumerate}
\end{lemma}  

For the proof of (A) we refer to~\cite[Lem.~3.1]{HeidSuli:21}. The assertion (B) can be shown in the same way as~\cite[Lem.~2]{HeidWihler2:19v1}, using the properties from Lemma~\ref{lem:anproperties} and noting that we only consider discretely divergence-free velocity vectors. Finally, the property~\eqref{eq:hnmono} follows as in the proof of~\cite[Thm.~2.5]{HeidWihler:19v2}, since the velocities involved are again discretely divergence-free. 

Now we are in a position to prove the convergence of the sequence of velocity vectors~$\{\Uv_N^{n,\ell}\}_{\ell}$ generated by~\eqref{eq:picardconfree} to $\uvn \in \V(\T_N)$. The following result, as well as its proof, are borrowed from~\cite[Thm.~1]{HeidPraetoriusWihler:21}; however, as the setting in this work is slightly different, we recall the proof of the statement. 

\begin{theorem} \label{thm:hncontraction}
Let $\Uv_N^n \in \V_0(\T_N)$ be the velocity vector of the  unique solution of~\eqref{eq:weakbinghamconfree} and let $\{\Uv_N^{n,\ell}\}_{\ell} \subset \V_0(\T_N)$ be the sequence generated by the iteration scheme~\eqref{eq:picardconfree}. Then, the following a posteriori error estimate holds:
\begin{align} \label{eq:aposterioriconfree}
\norm{\uvnl-\uvn}_{1,2} \leq \frac{\sigma n +2 \nu}{\nu} \norm{\uvnlo-\uvnl}_{1,2}, \qquad \ell=0,1,2,\dotsc.
\end{align}
Moreover, $\H_n$ contracts along the sequence generated by \eqref{eq:picardconfree}:
\begin{align} \label{eq:hncontraction}
\H_n(\uvnlo)-\H_n(\uvn) \leq q_{\rm con} (\H_n(\uvnl)-\H_n(\uvn)), 
\end{align} 
where 
\begin{align*} 
q_{\rm con}=\left(1-\frac{\nu^3}{(\sqrt{3} \sigma n+2 \nu)(\sigma n +2 \nu)^2}\right) \in (0,1).
\end{align*}
\end{theorem}

\begin{proof}
We first address the a posteriori error estimate. The strong monotonicity~\eqref{eq:anstrongmon} yields that
\begin{align*}
\nu \norm{\Uv_N^{n,\ell}-\Uv_N^n}_{1,2}^2 \leq a_n(\Uv_N^{n,\ell};\Uv_N^{n,\ell},\Uv_N^{n,\ell}-\uvn)-a_n(\uvn;\uvn,\uvnl-\uvn).
\end{align*} 
Then, since $\uvn$ and $\uvnl$ are discretely divergence-free, i.e., $b(P_N^n,\uvnl-\uvn)=b(Q,\uvn)=0$ for all $Q \in \Q(\T_N)$,~\eqref{eq:weakbinghamconfree} implies that $a_n(\uvn;\uvn,\uvnl-\uvn)=\int_\Omega \fv \cdot (\uvnl-\uvn) \dx$, and, in turn,
\begin{align*} 
\nu \norm{\uvnl-\uvn}_{1,2}^2 \leq a_n(\uvnl;\uvnl,\uvnl-\uvn)-\int_\Omega \fv \cdot (\uvnl-\uvn) \dx.
\end{align*}
Now recalling the iteration scheme~\eqref{eq:picardconfree} and that $\uvnlo$ is discretely divergence-free, this further leads to
\begin{align*}
\nu \norm{\uvnl-\uvn}_{1,2}^2 \leq a_n(\uvnl;\uvnl-\uvnlo,\uvnl-\uvn).
\end{align*}
Finally, by invoking the bound~\eqref{eq:anbounded}, we get 
\begin{align*}
\norm{\uvnl-\uvn}_{1,2} \leq \frac{\sigma n +2 \nu}{\nu} \norm{\uvnlo-\uvnl}_{1,2}.
\end{align*}
Next, we will establish the contraction of $\H_n$ along $\{\uvnl\}_\ell$. Indeed, we have that
\begin{align*}
\H_n(\uvnlo)-\H_n(\uvn)&=\H_n(\uvnl)-\H_n(\uvn)-(\H_n(\uvnl)-\H_n(\uvnlo)) \\
&\stackrel{\eqref{eq:hnmono}}{\leq} \H_n(\uvnl)-\H_n(\uvn)-\frac{\nu}{2}\norm{\uvnlo-\uvnl}^2_{1,2} \\
&\stackrel{\eqref{eq:aposterioriconfree}}{\leq}\H_n(\uvnl)-\H_n(\uvn)-\frac{\nu^3}{2(\sigma n +2 \nu)^2}\norm{\uvnl-\uvn}^2_{1,2} \\
&\stackrel{\eqref{eq:equiv}}{\leq}\H_n(\uvnl)-\H_n(\uvn)-\frac{\nu^3}{(\sqrt{3} \sigma n+2 \nu)(\sigma n +2 \nu)^2}(\H_n(\uvnl)-\H_n(\uvn)),
\end{align*}
which shows~\eqref{eq:hncontraction}.
\end{proof}

\begin{remark}
We emphasize that related results can be found in~\cite{Diening:2020} and~\cite{HeidSuli:21}. 
\end{remark}


\begin{corollary} \label{cor:velocityconvergence}
As before, let $\uvn$ be the velocity vector of the unique solution of~\eqref{eq:weakbinghamconfree} and let $\{\uvnl\}_\ell$ be the sequence generated by~\eqref{eq:picardconfree}. Then the following inequality holds:
\begin{align} \label{eq:kacanovvelocityconvergence}
\norm{\uvn-\uvnl}_{1,2} \leq \sqrt{\frac{\sqrt{3} \sigma n + 2\nu}{\nu}} q_{\rm con}^{\ell/2} \norm{\uvn-\Uv_N^{n,0}}_{1,2}.
\end{align}
Thus, since $q_{\rm con} \in (0,1)$, we have that $\uvnl \to \uvn$ for $\ell \to \infty$.
\end{corollary}

\begin{proof}
This follows immediately from Theorem~\ref{thm:hncontraction} and~\eqref{eq:equiv}.
\end{proof}

Next, we will establish the convergence of the sequence of iterations for the pressure.

\begin{corollary} \label{cor:pressureconvergence}
Let $(\Uv_N^n,\pn) \in \V(\T_N) \times \Q_0(\T_N)$ be the unique solution of~\eqref{eq:weakbinghamconfree}, and let $\{(\uvnl,\pnl)\}_{\ell}$ be the sequence generated by~\eqref{eq:picardconfree}. Then we have that
\begin{align} \label{eq:pressureconvergence}
\norm{\pn-P_N^{n,\ell+1}}_2 \leq \beta_2^{-1}(\sqrt{3} \sigma n \nu^{-1}+3)(\sigma n +2 \nu) \norm{\uvnlo-\uvnl}_{1,2},
\end{align}
and $\pnl \to \pn$ as $\ell \to \infty$.
\end{corollary}

\begin{proof}
Take any $\Vv \in \V(\T_N)$. Since both $\uvn$ and $\uvnlo$ are discretely divergence-free, it follows from~\eqref{eq:weakbinghamconfree} and~\eqref{eq:picardconfree} that
\begin{align*}
b(\pn-P_N^{n,\ell+1},\Vv)&=a_n(\uvnl;\uvnlo,\Vv)-a_n(\uvn;\uvn,\Vv) \\
&=a_n(\uvnl;\uvnlo-\uvnl,\Vv) + a_n(\uvnl;\uvnl,\Vv)-a_n(\uvn;\uvn,\Vv).
\end{align*} 
Invoking~\eqref{eq:anbounded},~\eqref{eq:anlipschitz}, and, subsequently,~\eqref{eq:aposterioriconfree}, we find that
\begin{align*}
b(\pn-\pnl,\Vv) &\leq (\sigma n + 2 \nu)\norm{\uvnlo-\uvnl}_{1,2}\norm{\Vv}_{1,2}+(\sqrt{3} \sigma n+ 2\nu)\norm{\uvnl-\uvn}_{1,2}\norm{\Vv}_{1,2} \\
&\leq (\sigma n + 2 \nu)\norm{\uvnlo-\uvnl}_{1,2}\norm{\Vv}_{1,2} \\
& \quad + (\sqrt{3} \sigma n+ 2\nu)\frac{\sigma n+2 \nu}{\nu} \norm{\uvnlo-\uvnl}_{1,2}\norm{\Vv}_{1,2} \\
&=(\sqrt{3} \sigma n \nu^{-1}+3)(\sigma n +2 \nu)\norm{\uvnlo-\uvnl}_{1,2}\norm{\Vv}_{1,2}.
\end{align*}
Finally, the discrete inf-sup condition~\eqref{eq:discreteinfsup} implies~\eqref{eq:pressureconvergence}. The convergence of the pressure sequence now follows immediately from Corollary~\ref{cor:velocityconvergence}.
\end{proof}

We remark that, thanks to Corollaries~\ref{cor:velocityconvergence} and~\ref{cor:pressureconvergence}, the Assumption~\ref{as:discreteresidual} is satisfied, since $\F_{N,pde}^{n}:\V(\T_N) \times \Q_0(\T_N) \to \V(\T_N)^\star$ is a continuous operator. Nonetheless, we will state a convenient upper bound for the discrete residual.

\begin{corollary} \label{cor:discreteresidualconvergence}
Let $(\Uv_N^n,\pn) \in \V(\T_N) \times \Q_0(\T_N)$ be the unique solution of~\eqref{eq:weakbinghamconfree}, and let $\{(\uvnl,\pnl)\}_{\ell}$ be the sequence generated by~\eqref{eq:picardconfree}. Then we have that
\begin{align*}
\norm{\F_{N,pde}^{n,\ell}}_{N,-1,2} &\leq \left(2 \sqrt{3} \sigma^2 n^2 \nu^{-1}+5 (\sqrt{3}+1) \sigma n + 12 \nu\right) \norm{\uvnl-\Uv_N^{n,\ell-1}}_{1,2},
\end{align*}
and 
\begin{align} \label{eq:icresidualbound}
\norm{\F_{N,ic}^{n,\ell}}_{N,-1,2}=0. 
\end{align}
Consequently,
\begin{align*} 
\norm{\F_{N,pde}^{n,\ell}}_{N,-1,2} + \norm{\F_{N,ic}^{n,\ell}}_{N,-1,2} \to 0 \qquad \text{for} \qquad \ell \to \infty.
\end{align*}
\end{corollary} 

\begin{proof}
Since $(\uvn,P_N^n) \in \V_0(\T_N) \times \Q_0(\T_N)$ is a solution of~\eqref{eq:weakbinghamconfree}, we find that, for any $\Vv \in \V(\T_N)$,
\begin{align*}
\dprod{\F_{N,pde}^{n}(\uvnl,\pnl),\Vv}&=\dprod{\F_{N,pde}^{n}(\uvnl,\pnl),\Vv}-\dprod{\F_{N,pde}^n(\uvn,P_N^n),\Vv}\\
&=a_n(\uvnl;\uvnl,\Vv)-a_n(\uvn;\uvn,\Vv) + b(P_N^{n,\ell}-P_N^n,\Vv).
\end{align*}
Hence, by~\eqref{eq:anlipschitz} and (the proof of) Corollary~\ref{cor:pressureconvergence}, it further follows that
\begin{align*}
\dprod{\F_{N,pde}^{n,\ell},\Vv} &\leq (\sqrt{3} \sigma n+ 2 \nu) \norm{\uvn-\uvnl}_{1,2} \norm{\Vv}_{1,2} \\
& \quad +(\sqrt{3}\sigma n \nu^{-1}+3)(\sigma n + 2 \nu)\norm{\uvnl-\Uv_N^{n,\ell-1}}_{1,2}\norm{\Vv}_{1,2}.
\end{align*}
Finally, recalling~\eqref{eq:aposterioriconfree} and employing the triangle inequality yields
\begin{align*}
\dprod{\F_{N,pde}^{n,\ell},\Vv} &\leq \left((\sqrt{3} \sigma n+2\nu)(\sigma n \nu^{-1}+3)+(\sqrt{3}\sigma n \nu^{-1}+3)(\sigma n + 2 \nu)\right)\norm{\uvnl-\Uv_N^{n,\ell-1}}_{1,2}\norm{\Vv}_{1,2},
\end{align*}
which implies the first estimate. Moreover,~\eqref{eq:icresidualbound} follows immediately from the definition of $\F^n_{N,ic}$, cf.~\eqref{eq:discreteresidualic}, and as $\uvnl$ is discretely divergence-free.
\end{proof}

\begin{remark}
The Ka\v{c}anov scheme~\eqref{eq:picardconfree} was already considered in the paper~\cite{Aposporidis:2011}, in which it was referred to as the Picard iteration. However, we (significantly) improved their convergence results, cf.~\cite[Prop.~4 \& Thm.~2]{Aposporidis:2011}. We further note that both of our convergence results, cf.~\eqref{eq:kacanovvelocityconvergence} and~\eqref{eq:pressureconvergence}, depend adversely on the regularisation parameter $n$; in particular, the convergence may slow down for increasing $n$, which was already observed in~\cite{GrinevichOlshanskii:09, Aposporidis:2011}. Just recently, this issue was addressed in~\cite{Pollock:2021}. Indeed, those authors proposed a nonlinear solver based on Anderson acceleration, introduced in~\cite{Anderson:1965}, applied to the iteration scheme~\eqref{eq:picardconfree} in order to accelerate the convergence. Moreover, this also lead to a method that is, to a certain extent, robust with respect to the regularisation parameter $n$.  
\end{remark}

\subsection{Zarantonello iteration for general Bingham fluids}

Now we will also incorporate the convection term. Restricting the test functions to $\V_0(\T_N)$, the problem reduces to finding $\Uv_N^n \in \V_0(\T_N)$ such that
\begin{align} \label{eq:bingone}
a_n(\Uv_N^n;\Uv_N^n,\Vv)+\B[\Uv_N^n;\Uv_N^n,\Vv]&=\int_\Omega \fv \cdot \Vv \dx \qquad \text{for all} \ \Vv \in \V_0(\T_N).
\end{align} 
Let us define, for any given $\Uv \in \V(\T_N)$, the linear and invertible\footnote{This follows from the Lax--Milgram theorem, since, for any $\Uv \in \V(\T_N)$, $a_n(\Uv;\cdot,\cdot)$ is coercive and bounded on $\V(\T_N) \times \V(\T_N)$.} operator $\Ao_N^n[\Uv]: \V(\T_N) \to \V(\T_N)^\star$ by
\begin{align} \label{eq:Aop}
\Ao_N^n[\Uv](\Vv)(\Wv):=a_n(\Uv;\Vv,\Wv) \qquad \text{for all} \ \Vv, \Wv \in \V(T_N),
\end{align} 
and $\F_N^n:\V(\T_N) \to \V(\T_N)^\star$ by
\begin{align} \label{eq:Fbing}
\dprod{\F_N^n(\Uv),\Vv}:=\Ao_N^n[\Uv](\Uv)(\Vv)+\mathcal{B}[\Uv;\Uv,\Vv]-\fv(\Vv),
\end{align}
where $\fv(\Vv):=\int_\Omega \fv \cdot \Vv \dx$ for all $\Vv \in \V(\T_N)$. Then, problem~\eqref{eq:bingone} can be stated equivalently as follows: find $\Uv_N^n \in \V_0(\T_N)$ such that 
\begin{align} \label{eq:bingoneF}
\F_N^n(\Uv_N^n)=0 \qquad \text{in} \ \V_0(\T_N)^\star,
\end{align}
and \eqref{eq:weakbinghampde} can be restated as
\begin{align} \label{eq:beq}
b(P_N^n,\Vv)=-\dprod{\F_N^n(\Uv_N^n),\Vv} \qquad \text{for all} \ \Vv \in \V(\T_N).
\end{align}
We will consider the following fixed-point iteration scheme for the solution of~\eqref{eq:bingoneF}:
\begin{align} \label{eq:zarantonello}
\Uv_N^{n,\ell+1}=\To_N^n(\Uv_N^{n,\ell}):=\uvnl-\delta \J^{-1}\F_N^n(\uvnl),
\end{align}
where $\Uv_N^{n,0} \in \V_0(\T_N)$ is an initial guess and $\delta >0$ is a damping parameter. Moreover, $\J:\V_0(\T_N) \to \V_0(\T_N)^\star$ denotes the isometric isomorphism between $\V_0(\T_N)$ and $\V_0(\T_N)^\star$ from the Riesz representation theorem, where $\V_0(\T_N)$ is endowed with the inner-product $
(\Uv,\Vv)=\int_\Omega \D \Uv:\D \Vv \dx$. We note that~\eqref{eq:zarantonello} is, in particular, the Zarantonello iteration, cf.~Zarantonello's original work~\cite{Zarantonello:60} or the monograph~\cite[\S 25.4]{Zeidler:90}. Moreover, the iteration scheme~\eqref{eq:zarantonello} can be stated equivalently as
\begin{align} \label{eq:zarantonelloweak}
\int_\Omega \D \Uv_N^{n,\ell+1}: \D \Vv \dx=\int_\Omega \D \uvnl : \D \Vv \dx-\delta \dprod{\F_N^n(\uvnl),\Vv} \qquad \text{for all} \ \Vv \in \V_0(\T_N).
\end{align}
We will now examine the convergence of this iteration scheme. First, we will establish the self-mapping property $\tn:B_{R}^{N,\uv} \to B_{R}^{N,\uv}$, where $B_R^{N,\uv}:=\{\Uv \in \V_0(\T_N): \norm{\D \Uv}_2 \leq R\}$ for a suitable choice of $R>0$. 

\begin{lemma} \label{lem:Tbounded}
Assume the small data property
\begin{align} \label{eq:smalldata}
\norm{\fv}_{\star} :=\sup_{\Vv \in \V_0(\T_N) \setminus\{\mathbf{0}\}} \frac{\fv(\Vv)}{\norm{\D\Vv}_2} < \frac{\nu^2}{2 \sqrt{2}C_\B}.
\end{align}
If $0<\delta \leq 2 (4\nu + \sigma n)^{-1}$, then we have that
\begin{align} \label{eq:selfmap}
\tn(B_R^{N,\uv}) \subseteq B_R^{N,\uv} \quad \text{for} \quad \frac{\nu-\sqrt{\nu^2-2 \sqrt{2} \Cb \norm{\fv}_{\star}}}{2 \sqrt{2} \Cb} \leq R \leq \frac{\nu +\sqrt{\nu^2-2 \sqrt{2}\Cb \norm{\fv}_{\star}}}{2 \sqrt{2} \Cb};
\end{align}
in particular, the radius $R$ is independent of $n,N \in \mathbb{N}$.
\end{lemma} 

\begin{proof}
By the definitions of the operators $\tn$, $\F_N^n$, and $\Ao_N^n$, cf.~\eqref{eq:zarantonello},~\eqref{eq:Fbing}, and \eqref{eq:Aop}, respectively, we have that
\begin{align*}
\norm{\D \tn(\Uv)}_2^2&= \int_\Omega \D \Uv:\D \tn(\Uv) \dx- \delta a_n(\Uv,\Uv,\tn(\Uv))-\delta \mathcal{B}[\Uv;\Uv,\tn (\Uv)]+\delta \fv(\tn(\Uv)) \\
&=\int_\Omega (1-\delta \mu_n(|\Du|^2))\D \Uv :\D \tn(\Uv) \dx-\delta \B[\Uv;\Uv,\tn(\Uv)]+\delta \fv(\tn(\Uv)).
\end{align*}
Thanks to~\eqref{eq:mukey} and the assumption on $\delta$ we further have that $|1-\delta \mu_n(|\Du|^2)| \leq (1-2 \delta \nu)$. Hence, the Cauchy--Schwarz inequality,~\eqref{eq:bbounded}, and Korn's inequality further yield that
\begin{align*}
\norm{\D \tn(\Uv)}_2 \leq (1-2 \delta \nu) \norm{\D \Uv}_2+2 \sqrt{2} \delta \Cb \norm{\D \Uv}_2^2 + \delta \norm{\fv}_{\star}.
\end{align*}
Consequently, in order to obtain a self-mapping, we need to find $R>0$ such that
\begin{align*}
(1-2 \delta \nu)R + 2 \sqrt{2} \delta C_\B R^2+ \delta \norm{\fv}_{\star} \leq R,
\end{align*}
which can simply be determined from the corresponding quadratic equation, whose solutions are given by the bounds in~\eqref{eq:selfmap}.
\end{proof}

Next, we will show that the operator $\F_N^n:\V_0(\T_N) \to \V_0(\T_N)^\star$ is strongly monotone and Lipschitz continuous on $B_R^\uv$, for $R$ as in~\eqref{eq:selfmap}, since this implies the convergence of the Zarantonello iteration scheme~\eqref{eq:zarantonello}.

\begin{theorem} \label{thm:velcon}
Suppose that the assumptions of Lemma~\ref{lem:Tbounded} hold. For $R >0$ as in~\eqref{eq:selfmap} we have that 
\begin{align} \label{eq:fstronglymonotone}
\dprod{\F_N^n(\Uv)-\F_N^n(\Vv),\Uv-\Vv} \geq \nu_\F(R) \norm{\Uv-\Vv}_{1,2}^2 \qquad \text{for all} \ \Uv,\Vv \in B_R^{N,\uv},
\end{align}
where $\nu_\F(R):=(\nu - \sqrt{2}\Cb R)>0$, and 
\begin{align} \label{eq:flipschitz}
|\dprod{\F(\Uv)-\F(\Vv),\Wv}| \leq L_\F (R,n) \norm{\Uv-\Vv}_{1,2} \norm{\Wv}_{1,2}  
\end{align}
for all $\Uv,\Vv \in B_R^{N,\uv}, \Wv \in \V_0(\T_N)$, where $L_\F(R,n):= (\sqrt{3} \sigma n+2 \nu +2 \sqrt{2} \Cb R)$. In particular, the iteration scheme~\eqref{eq:zarantonello} converges, for any initial guess $\Uv_N^{n,0} \in B_R^{N,\uv}$ and $0<\delta < 2 \nu_\F L_\F^{-2}$, to the unique solution $\Uv_N^n$ of~\eqref{eq:bingone} contained in the closed ball $B_R^{N,\uv}$.
\end{theorem}

\begin{proof}
For any $\Uv,\Vv \in B_R^{N,\uv}$ we have by the definition of $\F_N^n$ that 
\begin{align*}
\dprod{\F_N^n(\Uv)-\F_N^n(\Vv),\Uv-\Vv} &= a_n(\Uv;\Uv,\Uv-\Vv)-a_n(\Vv;\Vv,\Uv-\Vv) \\
& \quad + \B[\Uv;\Uv,\Uv-\Vv]-\B[\Vv;\Vv,\Uv-\Vv].
\end{align*}
We note that
\[
\B[\Uv;\Uv,\Uv-\Vv]-\B[\Vv;\Vv,\Uv-\Vv]=\B[\Uv-\Vv;\Uv,\Uv-\Vv]-\B[\Vv;\Vv-\Uv,\Uv-\Vv],
\]
and thus, by ~\eqref{eq:skewsym}, \eqref{eq:bbounded}, and~\eqref{eq:korn},
\begin{align*}
|\B[\Uv;\Uv,\Uv-\Vv]-\B[\Vv;\Vv,\Uv-\Vv]| \leq \sqrt{2} \Cb \norm{\D \Uv}_2 \norm{\Uv-\Vv}_{1,2}^2 \leq \sqrt{2} \Cb R \norm{\Uv-\Vv}_{1,2}^2. 
\end{align*}
This, together with \eqref{eq:anstrongmon}, immediately implies~\eqref{eq:fstronglymonotone}. We note that $\nu-\sqrt{2} \Cb R>0$ for any $R$ as in~\eqref{eq:selfmap}.

Next we will show the local Lipschitz continuity of $\F_N^n$. Take any $\Uv,\Vv \in B^{N,\uv}_R$ and $\Wv \in \V_0(\T_N)$. By the triangle inequality we have that
\begin{align*}
|\dprod{\F_N^n(\Uv)-\F_N^n(\Vv),\Wv}| &\leq \left|a_n(\Uv;\Uv,\Wv)-a_n(\Vv;\Vv,\Wv)\right| +\left|\mathcal{B}[\Uv;\Uv,\Wv]-\B[\Vv;\Vv,\Wv]\right|.
\end{align*}
Similarly as before, the second summand can be bounded by
\begin{align*}
|\mathcal{B}[\Uv;\Uv,\Wv]-\B[\Vv;\Vv,\Wv]| &\leq |\mathcal{B}[\Uv-\Vv;\Uv,\Wv]-\B[\Vv;\Vv-\Uv,\Wv]|\\
& \leq 2 \sqrt{2} \Cb R \norm{\Uv-\Vv}_{1,2}\norm{\Wv}_{1,2}.
\end{align*}
Together with the bound~\eqref{eq:anlipschitz} this implies that
\begin{align*}
|\dprod{\F_N^n(\Uv)-\F_N^n(\Vv),\Wv}| \leq (\sqrt{3} \sigma n+2 \nu +2 \sqrt{2} \Cb R)\norm{\Uv-\Vv}_{1,2}\norm{\Wv}_{1,2}.
\end{align*}

Finally, we note that $\delta < 2 \nu_\F(R) L_\F(R,n)^{-2} < 2(4 \nu+\sigma n)^{-1}$, and therefore Lemma~\ref{lem:Tbounded} implies that $\tn:B_R^{N,\uv} \to B_R^{N,\uv}$ is a self-map. Consequently, the strong monotonicity~\eqref{eq:fstronglymonotone} together with the Lipschitz continuity~\eqref{eq:flipschitz} imply that $\tn:B_R^{N,\uv} \to B_R^{N,\uv}$ is a contraction on $B_R^{N,\uv}$, see, e.g., the proof of~\cite[Thm.~25.B]{Zeidler:90}. Thus, by the Banach fixed-point theorem, the sequence generated by~\eqref{eq:zarantonello} converges to the unique solution $\Uv_N^n$ of~\eqref{eq:bingoneF} in $B_R^{N,\uv}$.  
\end{proof}

\begin{remark}
We note that a small data assumption similar to~\eqref{eq:smalldata} guarantees the existence of a \emph{unique} solution of the steady Navier--Stokes equation for Newtonian fluids, see, e.g.,~\cite[Ch.~4, Thm.~2.2]{GiraultRaviart:86}. Its proof is also based on the contraction of an iterative mapping, which, however, is different from the one stated in~\eqref{eq:zarantonello}. 
\end{remark}

\begin{corollary} \label{cor:binghamconv}
The problem~\eqref{eq:weakbingham} has a unique solution $(\Uv_N^n,P_N^n) \in \V(\T_N) \times \Q_0(\T_N)$ such that $\Uv_N^n \in B_R^{N,\uv}$. 
\end{corollary}

\begin{proof}
This follows immediately from Theorem~\ref{thm:velcon} and the (discrete) inf-sup condition~\eqref{eq:discreteinfsup}.
\end{proof}

Now let us consider the pressure term. To that end, we define the iteration scheme 
\begin{subequations} \label{eq:zarantonellop}
\begin{align}
\int_\Omega \D \Uv_N^{n,\ell+1}:\D \Vv \dx + \delta b(P_N^{n,\ell+1},\Vv)&=\int_\Omega \D \uvnl:\D \Vv \dx - \delta \dprod{\F_N^n(\uvnl),\Vv} \label{eq:zarantonelloppde} \\
b(\Uv_N^{n,\ell+1},Q)&=0
\end{align}
\end{subequations}
for all $(\Vv,Q) \in \V(\T_N) \times \Q(\T_N)$, where $\Uv_N^{n,0} \in \V(\T_N)$ is an arbitrary initial guess.

\begin{corollary}
Let $\Uv_N^{n,0} \in B_R^{N,\uv}$, where $R$ satisfies~\eqref{eq:selfmap}, and $0 < \delta < 2 \nu_\F L_\F^{-2}$, cf.~Theorem~\ref{thm:velcon}. Then, the sequence $\{(\Uv_N^{n,\ell},\pnl)\}_\ell$ generated by~\eqref{eq:zarantonellop} converges to the solution $(\Uv_N^n,P_N^n)$ from Corollary~\ref{cor:binghamconv}.
\end{corollary}

\begin{proof}
We note that the sequence of velocity vectors $\{\uvnl\}_\ell$ satisfies \eqref{eq:zarantonelloweak}, and thus, by Theorem~\ref{thm:velcon}, we have that $\Uv_N^{n,\ell}$ converges strongly to $\Uv_N^n$ in $W^{1,2}_0(\Omega)^d$. Moreover, since $(\Uv_N^n,P_N^n)$ is a solution of~\eqref{eq:weakbingham}, we further get that $b(P_N^n,\Vv)=-\dprod{\F_N^n(\Uv_N^n),\Vv}$ for all $\Vv \in \V(\T_N)$, cf.~\eqref{eq:beq}. Consequently, in view of~\eqref{eq:zarantonelloppde}, we obtain
\begin{align*}
\delta b(P_N^{n,\ell+1}-P_N^n,\Vv)&=\int_\Omega\D(\uvnl-\Uv_N^{n,\ell+1}):\D \Vv \dx+\delta \dprod{\F_N^n(\Uv_N^{n})-\F_N^n(\uvnl),\Vv}\\
& \leq \norm{\uvnl-\Uv_N^{n,\ell+1}}_{1,2}\norm{\Vv}_{1,2}+ \delta L_\F(R,n) \norm{\Uv_N^n-\uvnl}_{1,2}\norm{\Vv}_{1,2}.
\end{align*}
Then, the discrete inf-sup condition~\eqref{eq:discreteinfsup} implies that
\begin{align*}
\norm{P_N^{n,\ell+1}-P_N^n}_2 \leq (\beta_2 \delta)^{-1} \norm{\uvnl-\Uv_N^{n,\ell+1}}_{1,2}+\beta_2^{-1}L_\F(R,n) \norm{\Uv_N^n-\uvnl}_{1,2};
\end{align*}
together with the first part, i.e., the convergence of the sequence of the velocity vectors, this implies the convergence of the sequence of the pressure iterates. 
\end{proof}

\section{Numerical experiments}
Now we will perform two numerical experiments in order to highlight our analytical findings, one with and one without the convection term. Our algorithm is implemented in Python using the FEniCS software~\cite{fenics1, fenics2}. We will use the Taylor--Hood element pair for the discretisation, whereby our initial mesh consists of 32 uniform triangles, and, the initial velocity-pressure pair is chosen, in each case, to be the constant null function. For the graph approximation we will consider the mappings $\widetilde{\Sv}^m:=\Sv^{2^m}=\Sv^n$, where the latter is defined as in~\eqref{eq:graphapprox}; in the following, we call $m$ the graph approximation exponent and $n$ the graph approximation index. Then, it can be shown that
\begin{align*}
\mathcal{E}_\A(\Dv,\widetilde{\Sv}^m(\Dv))=\mathcal{E}_\A(\Dv,{\Sv}^{2^m}(\Dv)) \leq \frac{C(\sigma,\nu)}{2^{\nicefrac{2m}{3}}}=:\eta_{\A,m}(\Dv),
\end{align*}
where $C(\sigma,\nu)>0$ depends on the yield stress $\sigma$ and the viscosity $\nu$, cf.~\cite[\S 7]{KreuzerSuli:2016}; in our experiments below, we set $C(\sigma,\nu):=4$. We will select the elements for the refinement by the use of the D\"{o}rfler marking strategy, cf.~\cite{Doerfler:96}. Subsequently, the mesh is refined by the DOLFIN~\cite{dolfin1, dolfin2} subroutine \emph{refine}, which applies the Plaza algorithm~\cite{plaza:00}; we note that this refinement method is based on a number of bisections. The function $\zeta$ from~\eqref{eq:zetafunction}, which is of a rather theoretical purpose for our convergence analysis, is defined as $\zeta(N):=N^{-1}$. Moreover, after each mesh refinement or graph approximation update we perform at least one iteration step, independently of whether or not the criterion of the while loop, cf.~line 4 in Algorithm~\ref{alg:AILFEM}, is satisfied; of course, this does not interfere with the convergence of the algorithm.

\subsection{Bingham problem without the convection term} \label{ex:confree}

We will consider the Bingham fluid flow problem stated in~\cite[\S 6.1.1]{Aposporidis:2011}, see also~\cite[\S 5.2]{GrinevichOlshanskii:09}. We note that this is one of the very few problems for which the analytical solution is known. Here, the physical domain is given by $\Omega:=(0,1) \times (0,1)$, with Euclidean coordinates denoted by $(x,y)$, the yield stress is set to $\sigma:=0.3$, and the fluid viscosity is $\nu=1$. Then, the analytical solution for the velocity term is given by $\uv_\infty=(u_\infty,0)^{\mathrm{T}}$, where 
\begin{align} \label{eq:analyticalsolution}
u_\infty=
\begin{cases}
\frac{1}{8}\left(0.4^2-(0.4-2y)^2\right), & 0 \leq y \leq 0.2, \\
\frac{0.4^2}{8}, & 0.2 < y < 0.8, \\
\frac{1}{8}\left(0.4^2-(2y-1.6)^2 \right), & 0.8 \leq y \leq 1.
\end{cases}
\end{align}
Moreover, the Dirichlet boundary conditions are chosen accordingly to the solution $\uv_\infty$. 
In Figure~\ref{fig:adaptive_mesh} (left) we plot the error of the velocity vector $\norm{\uv_\infty-\Uv_N^{n_N,\ell_N}}_{1,2}$, as well as the a posteriori residual estimator\footnote{We note that this expression is indeed, up to a certain multiplicative constant, an upper bound on the residual $\mathcal{R}(\uvnll,\pnll,\tilde{\Sv}^{n_N}(\uvnll))$, cf.~Theorem~\ref{thm:upperbound}.} 
\begin{align} \label{eq:residualestimator}
\mathcal{E}_N\left(\uvnll,\pnll,\tilde{\Sv}^{n_N}(\uvnll)\right)^{\nicefrac{1}{2}}+\norm{\F_{N,pde}^{n_N}(\uvnll,\pnll)}_{N,-1,2},
\end{align}
against the number of elements in the mesh; here, $(\Uv_N^{n_N,\ell_N},P_N^{n_N,\ell_N})$ is the final iterate obtained by the Ka\v{c}anov iteration~\eqref{eq:picardconfree} on a given discrete space. We can see that the error of the velocity vector decays almost at a rate of $\mathcal{O}(|\T_N|^{-1})$, whilst the convergence rate of the residual estimator is $\mathcal{O}(|\T_N|^{-\nicefrac{1}{2}})$.  Furthermore, in Figure~\ref{fig:adaptive_mesh} (right), we visualise an intermediate mesh generated by the AILFEM. We observe that the mesh was mainly refined along the lines $y=0.2$ and $y=0.8$, i.e., the neighbourhood of the rigid regions, cp.~\eqref{eq:analyticalsolution}. In~\cite[Fig.~6.2]{Aposporidis:2011}, spikes of the error for the velocity approximation were observed in exactly those regions. Hence, this indicates that our mesh refinement strategy works well for the given example.

\begin{figure}
 \hfill
 \includegraphics[width=0.49\textwidth]{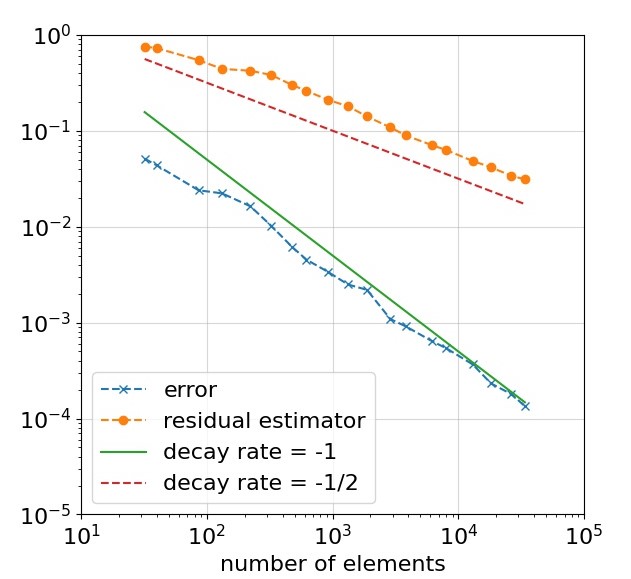}
\hfill
 \includegraphics[width=0.46\textwidth]{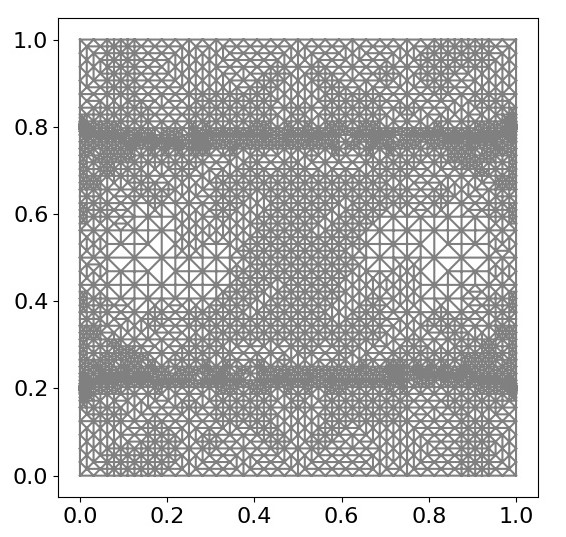}
 \hfill
 \caption{Ka\v{c}anov iteration for the Experiment~\ref{ex:confree}. Left: Convergence plot of the AILFEM. Right: Adaptively refined mesh with 8706 elements.}
 \label{fig:adaptive_mesh}
\end{figure}

In Table~\ref{table:noe} we list the number of iteration steps on each given Galerkin space, as well as the final graph approximation exponent $m$, against the number of elements in the mesh. Unfortunately, the number of iteration steps (slightly) growths with an increasing graph approximation exponent $m$ (and increasing number of elements in the mesh), which is in line with the estimator from Corollary~\ref{cor:discreteresidualconvergence}. However, we want to point out that the final graph approximation exponent $m=19$ leads to a corresponding index $n=2^{19} \approx 5 \cdot 10^5$ in~\eqref{eq:graphapprox}. In contrast, we note that the nonlinear variational Newton solver from FEniCS with default settings already failed to converge (in 50 steps) for $n=100$. Indeed, it was already noted in~\cite{DeanGlowinski:02} that, in the given setting, the domain of convergence for the Newton solver shrinks at a rate $\mathcal{O}(n^{-1})$.

\begin{table}
\begin{center}
\small{
  \begin{tabular}{c|cccccccccccc} 
    \toprule
    noe & $32$ &$40$ & $86$ & $132$ &$222$ &$324$&$478$&$620$&$922$&$1326$&$1878$&$2884$ \\ \midrule
nit &6& 1 &3&1&2&1&2&2&2&2&1&6\\
$m$ &5&5&7&7&8&8&9&10&11&12&12&14\\
	\bottomrule
    \end{tabular}
    }
    
\small{
  \begin{tabular}{c|ccccccccc} 
    \toprule
   noe &$3812$&$6162$ &8012&13044&18210&26476&33868 \\ \midrule
nit&2&8&8&14&24&19&23\\
$m$ &14&15&16&17&18&18&19\\
	\bottomrule 
    \end{tabular}
    }
      \captionof{table}{Ka\v{c}anov iteration for the Experiment~\ref{ex:confree}. The number of iterations (nit) on the given discrete space, as well as the final graph approximation exponent $m$, against the number of elements (noe) in the mesh.}
\label{table:noe}
\end{center} 
\end{table}



\subsection{Bingham problem with the convection term} \label{ex:convection}

For the Bingham fluid flow problem with the convection term we set the right-hand side function to 
\begin{align*}
\fv:=\begin{pmatrix}
\sin(\pi x) \cos(\pi y)-\cos(\pi x) \sin(\pi y) \\
xy
\end{pmatrix},
\end{align*}
where as before $(x,y) \in \Omega:=(0,1)^2$ denote the Euclidean coordinates, and impose homogeneous Dirichlet boundary conditions. Again, we consider the yield stress $\sigma=0.3$, and the fluid viscosity $\nu=1$. In the case of the Zarantonello iteration, we choose the damping parameter adaptively to be $\delta(n):=n^{-1}=2^{-m}$. In Figure~\ref{fig:convection} we plot the a posteriori residual estimator~\eqref{eq:residualestimator} against the number of elements in the mesh; even though, in presence of the convection term, we have guaranteed convergence for the Zarantonello iteration only, we will also consider the Ka\v{c}anov scheme:
\begin{align} \label{eq:kacanov}
a_n(\Uv_N^{n,\ell};\Uv_N^{n,\ell+1},\Vv)+\B[\Uv_N^{n,\ell};\Uv_N^{n,\ell+1},\Vv]+b(P_N^{n,\ell+1},\Vv)+b(Q,\Uv_N^{n,\ell+1})=\int_\Omega \fv \cdot \Vv \dx,
\end{align}
for all $(\Vv,Q) \in \V(\T_N) \times \Q(\T_N)$. As we can observe, the residual estimator decays at a rate $\mathcal{O}(|\T_N|^{-1/2})$ for both methods. Moreover, the Zarantonello scheme performed a total of 414 iteration steps, whereas the Ka\v{c}anov procedure only required 58 iteration steps. In both cases, the final graph approximation exponent was $m=23$. We note that the application of the Anderson acceleration, as proposed in~\cite{Pollock:2021} for the convection-free problem, might also enhance the convergence in the case that the convection term is included; however, this question is beyond the scope of our work considered here.

\begin{figure}
 \hfill
 \includegraphics[width=0.49\textwidth]{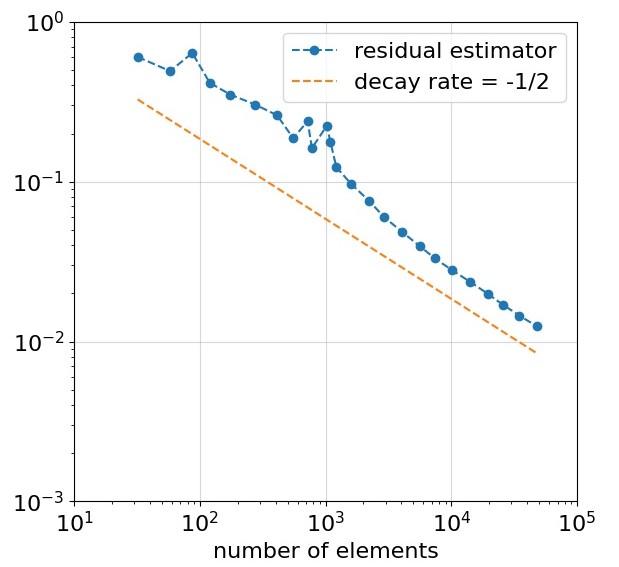}
\hfill
 \includegraphics[width=0.49\textwidth]{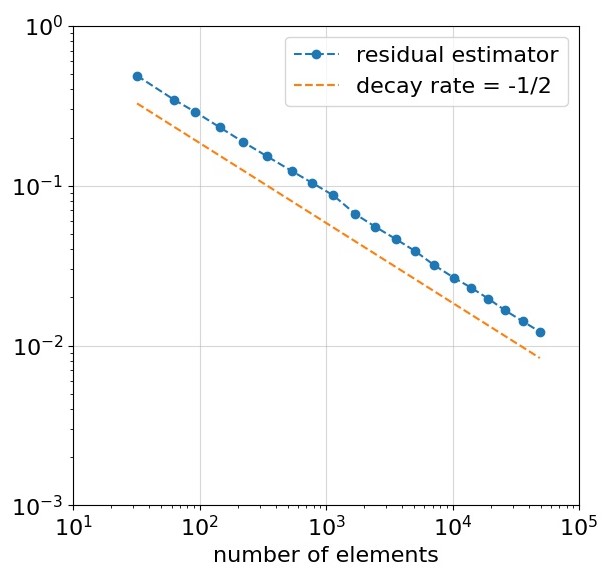}
 \hfill
 \caption{Experiment~\ref{ex:convection}. Left: Convergence plot for the Zarantonello iteration. Right: Convergence plot for the Ka\v{c}anov scheme.}
 \label{fig:convection}
\end{figure}
 
\begin{remark}
We note that the iteration scheme~\eqref{eq:kacanov} is also known as Picard's iteration, especially in the case of a constant viscosity, i.e., in the context of Newtonian fluids. Indeed, in~\cite[Ch.~4]{GiraultRaviart:86}, the uniqueness of the solution of the steady-state Navier--Stokes equation for Newtonian fluids under a small data assumption was proved by the contraction property of the fixed-point iteration corresponding to the Picard iteration. However, that analysis cannot be generalised to the setting considered here, as the bound on the source term $\fv$ depends in an unfavourable manner on the graph approximation index $n$; in particular, for $n \to \infty$, the only possible choice for the data is $\fv=\mathbf{0}$. 
\end{remark}
 
 \section{Conclusion}
In this work, in the context of implicitly constituted fluid flow problems, we further developed the adaptive finite element algorithm from~\cite{KreuzerSuli:2016} by taking into account the nonlinear solver, leading to the adaptive iterative linearised finite element algorithm introduced here. We showed that the Ka\v{c}anov and Zarantonello methods satisfy the assumptions on the nonlinear solver for the convergence of the AILFEM for Bingham fluids without and with inclusion of the convection term, respectively. Our numerical results indicate that the Ka\v{c}anov method does also convergence in the presence of the convection term, even with a significantly smaller number of iteration steps compared to the Zarantonello iteration, at least in the case of the specific model problem considered.

\appendix
\section{} \label{app:1}
In this appendix we will give a very rough sketch of the proof of Theorem~\ref{thm:main}. We emphasise once more that this result can be verified by some minor modifications, which will be stressed below, in the analysis of~\cite{KreuzerSuli:2016}.\\

For the purpose of the proof, we will define, for $s \in [1,\infty]$, the spaces
\[
\V_\infty^s:=\overline{\bigcup_{N \geq 0} \V(\T_N)}^{\norm{\cdot}_{1,s}} \subset W_0^{1,s}(\Omega)^d \qquad \text{and} \qquad \Q_{0,\infty}^s:=\overline{\bigcup_{N \geq 0} \Q_{0}(\T_N)}^{\norm{\cdot}_s} \subset L^s_0(\Omega).
\] 
We will now establish the convergence in several steps. \\
 
\textbf{Step 1 (\hspace{1sp}\cite[Lem.~5.2]{KreuzerSuli:2016}):}  Show that, for at least a not relabelled subsequence, we have that
\begin{align*}
\Uv_N^{n_N,\ell_N} &\rightharpoonup \uv_\infty && \text{weakly in} \ W_0^{1,r}(\Omega)^d,&\\
P_{N}^{n_N,\ell_N} &\rightharpoonup p_\infty && \text{weakly in} \ L_0^{\tilde{r}}(\Omega),&\\
\Sv^{n_N}(\D \Uv_N^{n_N,\ell_N}) & \rightharpoonup \Sv_\infty && \text{weakly in} \ L^{r'}(\Omega;\Rsym),&
\end{align*} 
for some $(\uv_\infty,p_\infty,\Sv_\infty) \in \V_\infty^r \times \Q_{0,\infty}^{\tilde{r}} \times L^{r'}(\Omega;\Rsym)$, as $N \to \infty$. Then, further verify that
\begin{align}
\R^{pde}(\Uv_N^{n_N,\ell_N},P_N^{n_N,\ell_N},\Sv^{n_N}(\D \Uv_N^{n_N,\ell_N})) &\rightharpoonup^\star \R^{pde}(\uv_\infty,p_\infty,\Sv_{\infty}) &&\text{weakly* in} \ W^{-1,\tilde{r}}(\Omega)^d, \label{eq:weakstarconvergence}\\
\R^{ic}(\Uv_N^{n_N,\ell_N}) &\rightharpoonup \R^{ic}(\uv_\infty) && \text{weakly in} \ L^{r}(\Omega), \label{eq:weakconvergenceic}
\end{align}
and 
\begin{align} \label{eq:step1sol}
\dprod{\R(\uv_\infty,p_\infty,\Sv_\infty),(\vv,q)}=0 \qquad \text{for all} \ q \in \Q_{0,\infty}^{r'}, \vv \in \V_\infty^{\tilde{r}'}.
\end{align}
\\
A key ingredient of the proof of those properties is the uniform bound~\eqref{eq:uniformbound} from
Lemma~\ref{lem:uniformbound}, as this yields the existence of weak limit points $\uv_\infty \in \V_\infty^r$ and $\Sv_\infty \in L^{r'}(\Omega;\Rsym)$. Then, in order to show that the weak limit $\uv_\infty \in \V_\infty^r$ is weakly divergence-free with respect to $\Q_{0,\infty}^{\tilde{r}}$, we have to employ an interpolation error bound for the operator $\jq$ and use that $\norm{\F_{N,ic}^{n_N,\ell_N}}_{N,-1,r} \to 0$ as $N \to \infty$ by line 4 in Algorithm~\ref{alg:AILFEM} and~\eqref{eq:zetafunction}. The weak convergence of the pressure iterates can again be shown via the boundedness of this sequence. In turn, for the sake of proving the uniform boundedness of the sequence of pressure iterates we have to employ the discrete inf-sup condition~\eqref{eq:discreteinfsup}. We further note that, compared to the analysis in~\cite{KreuzerSuli:2016}, we have in this context an additional term $\dprod{\F_{N,pde}^{n_N,\ell_N},\Vv}$, $\Vv \in \V(\T_N)$, which, however, can be bounded uniformly in $N$ by 
\begin{align*}
|\dprod{\F_{N,pde}^{n_N,\ell_N},\Vv}| \leq C_r \zeta_{\max} \norm{\Vv}_r,
\end{align*}
where $\zeta_{\max}:=\max_{N \in \mathbb{N}}\zeta(N)$, cf.~line 4 in Algorithm~\ref{alg:AILFEM} and~\eqref{eq:zetafunction}, and $C_r$ is the constant from Remark~\ref{rem:dualnormequiv}. The remaining properties~\eqref{eq:weakstarconvergence}--\eqref{eq:step1sol} can be established as in~\cite{KreuzerSuli:2016} using the density $\V_\infty^\infty \subseteq \V_\infty^{\tilde{r}'}$, and, on account of inexact finite element approximations, since $\dprod{\F_{N,pde}^{n_N,\ell_N},\jdiv \Vv}$ vanishes as $N \to \infty$. \\

In the following, $\{(\Uv_N^{n_N,\ell_N},P_N^{n_N,\ell_N},\Sv^{n_N}(\D \Uv_N^{n_N,\ell_N}))\}_{N} \subset W_0^{1,r}(\Omega)^d \times L_0^{\tilde{r}}(\Omega) \times L^{r'}(\Omega;\Rsym)$ will always denote a not relabelled subsequence with weak limit $(\uv_\infty,p_\infty,\Sv_\infty) \in \V_\infty^r \times \Q_{0,\infty}^{\tilde{r}} \times L^{r'}(\Omega;\Rsym)$ as obtained in Step~1. We will show below, cf.~Steps 2--4, that $(\uv_\infty,p_\infty,\Sv_\infty)$ is a solution of~\eqref{eq:pde} by verifying that the equivalent properties~\eqref{eq:equivprop} from Lemma~\ref{lem:solution} are satisfied. \\

\textbf{Step 2 (\hspace{1sp}\cite[Cor.~5.3]{KreuzerSuli:2016}):}  Show that 
\begin{align} \label{eq:step2ass}
\mathcal{E}_N(n_N,\ell_N) \to 0 \qquad \text{as} \ N \to \infty 
\end{align} 
implies that
\begin{align*} 
\R(\uv_\infty,p_\infty,\Sv_\infty)=0 \qquad \text{in} \ W^{-1,\tilde{t}}(\Omega)^d \times L_0^t(\Omega).
\end{align*}
\\
Indeed, this claim follows from Theorem~\ref{thm:upperbound}, the weak* convergence from~\eqref{eq:weakstarconvergence}, the weak convergence from~\eqref{eq:weakconvergenceic} (with respect to $L^t(\Omega)$), the fact that $\norm{\F_{N,pde}^{n_N,\ell_N}}_{N,-1,\tilde{t}} \to 0$ as $N \to \infty$, and the uniqueness of the limit point.\\

\textbf{Step 3 (\hspace{1sp}\cite[Lem.~5.4]{KreuzerSuli:2016}):} The next step is to verify that 
\begin{align} \label{eq:step3ass}
\mathcal{E}_\A(N,n_N,\ell_N) \to 0 \qquad \text{as} \ N \to \infty
\end{align}
yields
\begin{align*} 
(\D \uv_\infty(\x),\Sv_\infty(\x)) \in \A(\x) \qquad \text{for almost every} \ \x \in \Omega.
\end{align*}
In turn, by definition of $\mathcal{E}_\A$, cf.~\eqref{eq:graphapproxerror}, we have that $\mathcal{E}_\A(\D \uv_\infty,\Sv_\infty)=0$.\\

\noindent The proof of this statement is quite lengthy and technical, and we shall simply refer to the corresponding proof of~\cite[Lem.~5.4]{KreuzerSuli:2016}. We note, however, that on account of inexact finite element approximations $\Uv_N^{n_N,\ell_N} \in \V(\T_N)$ of~\eqref{eq:discreteproblem} there appears an additional term $\dprod{\F_{N,pde}^{n_N,\ell_N},\mathit{\Phi}_{N,j}}$, $N,j \in \mathbb{N}$, in the proof\footnote{In particular, this additional term occurs in the summand $\Rnum{2}_{k,j}$, for $k=N$, of Step 2 in the proof of~\cite[Lem.~5.4]{KreuzerSuli:2016} on page 1359.}, where $\{\mathit{\Phi}_{N,j}\}_{N}$ is uniformly bounded in $W_0^{1,r}(\Omega)^d$. Consequently,
\[
|\dprod{\F_{N,pde}^{n_N,\ell_N},\mathit{\Phi}_{N,j}}| \leq \norm{\F_{N,pde}^{n_N,\ell_N}}_{N,-1,r'} \norm{\mathit{\Phi}_{N,j}}_{1,r} \to 0 \qquad \text{as} \ N \to \infty,
\]   
since $\norm{\F_{N,pde}^{n_N,\ell_N}}_{N,-1,r'}\leq C_r \zeta(N) \to 0$ as $N \to \infty$ by the modus operandi of AILFEM,~\eqref{eq:zetafunction}, and Remark~\ref{rem:dualnormequiv}.\\

Hence, in view of Lemma~\ref{lem:solution} and Steps 2 and 3 from above it only remains to establish~\eqref{eq:step2ass} and~\eqref{eq:step3ass}.\\

\textbf{Step 4 (\hspace{1sp}\cite[Thm.~5.7]{KreuzerSuli:2016}):} Show that~\eqref{eq:step3ass} holds for the subsequence from Step 1, and~\eqref{eq:step2ass} for a sub-subsequence of that subsequence.\\

\noindent In order to prove those two assertions, we will need the following preliminary result.

\begin{lemma}[{\hspace{1sp}\cite[Lem.~5.5]{KreuzerSuli:2016}}] \label{lem:app} If $n_N=\tilde{n}$ for some $\tilde{n} \in \mathbb{N}$ and all $N$ large enough, then we have that, for at least a not relabelled sub-subsequence, 
\[
\mathcal{E}_N(n_N,\ell_N) \to 0 \qquad \text{as} \ N \to \infty.
\] 
\end{lemma} 

The proof of this lemma is again very technical and long. However, up to some minor modifications, it coincides with the corresponding proof in~\cite{KreuzerSuli:2016}. Indeed, we only have to deal with some additional terms of the form $\dprod{\F_{N,pde}^{n_N,\ell_N},\Vv_N}$ and $\dprod{\F_{N,ic}^{n_N,\ell_N},Q_N}$, where $\Vv_N \in \V(\T_N)$ and $Q_N \in \Q(\T_N)$ are uniformly bounded in $W^{1,s}(\Omega)^d$, $s \in \{r,\tilde{t}'\}$, and $L^{s}(\Omega)$, $s \in \{r',t'\}$, respectively. Hence, those terms vanish as $N$ goes to infinity thanks to the employed stopping criterion of the while-loop (lines 4--6) in Algorithm~\ref{alg:AILFEM} and Remark~\ref{rem:dualnormequiv}.\\

Now we will use Lemma~\ref{lem:app} to prove Step 4, which will be done along the lines of the proof of~\cite[Thm.~5.7]{KreuzerSuli:2016}. First, assume for the sake of contradiction that there exists a positive constant $\epsilon>0$ such that, for some not relabelled sub-subsequence,
\begin{align*}
\eta_\A(N,n_N,\ell_N) \geq \epsilon \qquad \text{for all} \ N=0,1,2,\dotsc.
\end{align*}
In view of Assumption~\ref{as:graphapp2} this implies that $n_N=\tilde{n}$ for some $\tilde{n} \in \mathbb{N}$ and all $N$ large enough. Consequently, Lemma~\ref{lem:app} yields that, for a further not relabelled sub-subsequence, $\mathcal{E}_N(n_N,\ell_N) \to 0$ as $N \to \infty$. In particular, for some $N$ large enough we have that $\mathcal{E}_N(n_N,\ell_N)<\epsilon$, and thus, by lines 9--14 of Algorithm~\ref{alg:AILFEM}, we set $n_{N+1}=n_N+1=\tilde{n}+1$, which yields the desired contradiction; since $\mathcal{E}_{\A}(N,n_N,\ell_N) \leq \eta_\A(N,n_N,\ell_N)$, we have established property~\eqref{eq:step3ass}.

Next, we assume again for the sake of contradiction that there exists an $\epsilon>0$ such that
\begin{align} \label{eq:contradiciton}
\mathcal{E}_N(n_N,\ell_N)\geq \varepsilon \qquad \text{for all} \ N=0,1,2,\dotsc.
\end{align}
From the first part we know that there exists an integer $N_0$ such that $\eta_\A(N,n_N,\ell_N)<\epsilon \leq \mathcal{E}_N(n_N,\ell_N)$ for all $N \geq N_0$. Consequently, again by the modus operandi of AILFEM, we have that $n_N=n_{N_0}$ for all $N \geq N_0$. However, then Lemma~\ref{lem:app} contradicts~\eqref{eq:contradiciton}, and thus, at least for a not relabelled sub-subsequence,~\eqref{eq:step2ass} is satisfied.

\bibliographystyle{amsplain}
\bibliography{references}
\end{document}